%% file: paper_revised.tex
\documentclass[12pt]{article}

\input{header}

\title{Efficient TV regularization of large-scale linear inverse problems via the SCD semismooth$^*$ Newton method with applications in tomography}
\author{
Helmut Gfrerer\footnote{Johann Radon Institute for Computational and Applied Mathematics (RICAM), Altenbergerstra{\ss}e 69, A-4040 Linz, Austria (helmut.gfrerer@ricam.oeaw.ac.at)}\,\,\textsuperscript{,}\footnote{ Institute of Information Theory and Automation, Czech Academy of
Sciences, 18208 Prague, Czech Republic (gfrerer@utia.cas.cz)} ,
Simon Hubmer\footnote{Johannes Kepler University Linz, Institute of Industrial Mathematics, Altenbergerstra{\ss}e 69, A-4040 Linz, Austria, (simon.hubmer@jku.at), \textbf{Corresponding author}} ,
Stefan Kindermann\footnote{Johannes Kepler University Linz, Institute of Industrial Mathematics, Altenbergerstra{\ss}e 69, A-4040 Linz, Austria,
(kindermann@indmath.uni-linz.ac.at).}
\\
Jaakko Kultima\footnote{Johann Radon Institute Linz, Altenbergerstra{\ss}e 69, A-4040 Linz, Austria, (jaakko.kultima@ricam.oeaw.ac.at)} ,
Ronny Ramlau\footnote{Johannes Kepler University Linz, Institute of Industrial Mathematics, Altenbergerstra{\ss}e 69, A-4040 Linz, Austria, (ronny.ramlau@jku.at)}\,\,\textsuperscript{,}\footnote{
Johann Radon Institute Linz, Altenbergerstra{\ss}e 69, A-4040 Linz, Austria, (ronny.ramlau@ricam.oeaw.ac.at)}  ,
Tanja Tarvainen\footnote{University of Eastern Finland, Department of Technical Physics, 70211 Kuopio, Finland, (tanja.tarvainen@uef.fi)} ,
}

\begin{document}

\maketitle

\begin{abstract}
In this paper, we consider the efficient numerical minimization of Tikhonov functionals resulting from total-variation (TV) regularization of linear inverse problems. Since the TV penalty is non-smooth, this is typically done either via smooth approximations, which are inexact, or using non-smooth optimization techniques, which can often be numerically expensive, in particular for large-scale problems. Here, we present a numerically efficient minimization approach based on the recently proposed \ssstar Newton method, which employs a novel concept of graphical derivatives and exhibits locally superlinear convergence. The proposed approach is specifically tailored to TV regularization, suitable for large-scale inverse problems, and supported by strong mathematical convergence guarantees. Furthermore, we demonstrate its performance on two (large-scale) tomographic imaging problems and compare our results to those obtained via other state-of-the-art TV regularization approaches.


\end{abstract}


\section{Introduction}\label{sect_introduction}

In this paper, we consider (large-scale) linear inverse problems \cite{Engl_Hanke_Neubauer_1996} of the form
    \begin{equation}\label{Ax=b}
        A x = b \,, 
        \qquad
        \text{where}
        \qquad
        A : D(A) \subset X \to Y \,,
    \end{equation}
with $X$ and $Y$ being Banach or Hilbert spaces to be specified below, and their solution using total-variation (TV) regularization \cite{Scherzer_Grasmair_Grossauer_Haltmeier_Lenzen_2008}. In particular, let
    \begin{equation}\label{def_BV}
        \abs{x}_\BV := \sup_{ \phi \in C_0^\infty(\Omega;\R^n), \norm{\phi}_\infty \leq 1 } 
        \int_\Omega x(s) \div \phi(s) \, d s \,. 
    \end{equation}
denote the total variation of $x \in L^1(\Omega)$, $\Omega \subset \R^n$ (see, e.g., \cite{AmFuPa,Me,ChCaCrNoPo}), and assume that
    \begin{equation*}
        \norm{b - b^\delta} \leq \delta \,,
    \end{equation*}
for some noise level $\delta \geq 0$ and noisy data $b^\delta$, which are given instead of the true data $b$. Then classic TV regularization for the inverse problem \eqref{Ax=b} consists in computing \cite{Scherzer_Grasmair_Grossauer_Haltmeier_Lenzen_2008}
    \begin{equation}\label{TV_regularization}
        \xad := \argmin_{x \in D(A)} \mathcal{T}_\alpha^\delta(x) \,,
        \qquad \text{where} \qquad
        \mathcal{T}_\alpha^\delta(x) := \frac{1}{2} \norm{Ax-b^\delta}^2 +\alpha \abs{x}_\BV \,,
    \end{equation}
as an approximation to the minimum-norm solution $\xD$. For $A=I$, $\mathcal{T}_\alpha^\delta$ corresponds to the well-known Rudin-Osher-Fatemi denoising functional \cite{RuOsFa}, which pioneered the use of TV regularization in imaging. Since then, the penalty term $\abs{\cdot}_\BV$ has been used in countless applications \cite{Scherzer_Grasmair_Grossauer_Haltmeier_Lenzen_2008,Mueller_Siltanen_2012}, in particular due to its ability to induce piecewise constant reconstructions favored/required in many imaging problems.

However, while TV regularization is popular in applications, the non-smoothness of the penalty term $\abs{\cdot}_\BV$ also causes several practical difficulties, in particular relating to the minimization of $\mathcal{T}_\alpha^\delta$. A popular approach is to use smooth approximations such~as
    \begin{equation}\label{approx}
        \abs{x}_\BV \approx \int_\Omega \sqrt{ \abs{\nabla x(s)}^2 + \epsilon  } \, ds \,,
    \end{equation}
where $\epsilon$ is a small smoothing parameter \cite{RuOsFa,AcVo}, and then to apply standard techniques from (infinite-dimensional) smooth optimization. This is also a common route in finite-dimensional realizations of TV regularization, where in the above approximation the gradient is typically replaced by a suitable difference quotient. While these approximations often yield reasonable reconstructions, they are by their very nature inexact and require a tuning of the smoothing parameter $\epsilon$. Hence, another common route for the minimization of $\Tad$ is to consider the first-order optimality condition 
    \begin{equation}\label{first_order}
        \partial \, \Tad(x) = A^*(A x - b^\delta) + \alpha \partial \abs{x}_\BV \ni 0 \,,
    \end{equation}
which forms the basis of several minimization algorithms (such as the Chambolle-Pock method \cite{ChPo} and the ADMM Algorithm~\cite{BoydADMM}) reviewed below. While these approaches are free of approximations, their numerical application is typically quite computationally expensive; and prohibitively so for large-scale inverse problems. This is in part due to structural difficulties inherent in the $\abs{\cdot}_\BV$ penalty, but also due to the first-order nature of these approaches, which consequently require a large number of iterations, even when acceleration schemes such as Nesterov acceleration \cite{Nesterov_1983} are used. 

Hence, in this paper, we propose a new and efficient approach for TV regularization based on the recently proposed \ssstar Newton method \cite{Gfr25a,GfrOut21}, applicable also to large-scale inverse problems. The \ssstar Newton method was originally designed to efficiently solve set-valued inclusions of the form $0 \in F(x)$, where $F: \R^N \tto\R^N$ is a set-valued mapping. Using the novel concept of the subspace-containing derivative (SCD), a form of graphical derivative applicable in the non-smooth case, the method is essentially a second-order method for non-smooth problems and was shown to exhibit locally superlinear convergence. In \cite{Gfrerer_Hubmer_Ramlau_2025}, the \ssstar Newton method was used to minimize general variational regularization functionals of the form
    \begin{equation*}
        \Tad \, : \, \R^n \to \R^m \,, \qquad \Tad := S(G(x),b^\delta) + \alpha R(x)  \,,
    \end{equation*}
by essentially applying it to the first-order optimality condition, compare with \eqref{first_order},
    \begin{equation*}
        0 \in \partial \, \Tad = \partial \, S(G(x),b^\delta) + \alpha \partial \, R(x)  \,,
    \end{equation*}
where $\partial$ may here also denote the limiting subdifferential \cite{RoWe98} generalizing the classical subdifferential \cite{Bauschke_Combettes_2017}. Furthermore, \cite{Gfrerer_Hubmer_Ramlau_2025} also discussed the application of the \ssstar Newton method to a specific type of discretized TV regularization functional, namely
    \begin{equation}\label{TV_reg_old}
        \Tad (x) := \frac{1}{2} \norm{Ax-b^\delta}_{\R^m} ^2 +\alpha \sum_{i=1}^{n_2-1} \sum_{j=1}^{n_1-1} \abs{x_{i,j+1} - x_{i,j}} + \abs{x_{i+1,j}-x_{i,j}}\,,
    \end{equation}
where $A$ is an $m\times n$ matrix, $b^\delta \in \R^m$, $\alpha > 0$, and $(x_{i,j})$, $j=1,\dots,n_1$, $i=1,\dots,n_2$, is a matrix representation of the vector $x \in \R^n$. Here, the regularization penalty is a numerical approximation of $\abs{x}_\BV \approx \norm{\nabla x}_{L^1(\Omega)}$ commonly used in applications. More generally, one may consider the finite-dimensional Tikhonov regularization approach
    \begin{equation}\label{TV_reg_discr}
        \xad := \argmin_{x \in \R^n} \Tad(x) \,,
        \qquad \text{where} \qquad
        \Tad(x) := \frac{1}{2} \norm{Ax-b^\delta}_{\R^m}^2 +\alpha \norm{Bx}_1\,,
    \end{equation}
where now $B$ is an $l \times n$ matrix\footnote{The method proposed in this paper is not intended for the special case when $B$ is the identity matrix, for which more effective minimization methods are available.}. After discretization, \eqref{TV_regularization} typically leads to just such a finite-dimensional problem, where $B$ is then some type of discrete gradient matrix \cite{Mueller_Siltanen_2012}. Note that both matrices $A$ and $B$ may be very large and that $A$ may only be available via a routine for evaluating the matrix-vector products $Ax$ and $A^T b$. 
An Augmented Lagrangian method (ALM) was considered in \cite{Gfrerer_Hubmer_Ramlau_2025} for solving \eqref{TV_reg_discr}, where each of the subproblems was solved ``exactly'' by means of the \ssstar Newton method introduced in \cite{Gfr25a}. Although this approach did not  leverage the special structure of the problem, it already worked pretty well. The aim of this paper is to consider several modifications which lead to a further increase in computational efficiency. In particular, in the ALM we solve the subproblems only approximately by means of a regularized semismooth* Newton method. These subproblems require the minimization of a $C^{1,1}$-function involving the Moreau envelope of the nonsmooth term, and the use of semismooth Newton methods obeying locally superlinear convergence is widespread. The proposed regularization technique aims to faster reach the vicinity of the solution where superlinear convergence arises, and is based on the theory of SC {(\em subspace containing)} derivatives introduced in \cite{GfrOut22}. We prove that locally superlinear convergence is retained by this modification. We mention at this point that the convergence theory relies to some extent on the fact that the subgradient mapping $\partial T_\alpha^\delta$ is a polyhedral multifunction (i.e., its graph is the union of finitely many convex polyhedral sets). Hence, the obtained theoretical results cannot be carried over to the nonlinear case directly.

The resulting method is applicable also to large-scale inverse problems, i.e., large $m$ and $n$, with both global and local(ly superlinear) convergence guarantees.
Furthermore, we present a number of numerical experiments on two (large-scale) inverse problems from computerized and photoacoustic tomography and compare the results to those obtained via other state-of-the-art TV regularization approaches.
 
The outline of this paper is as follows: In Section~\ref{sect_TV_background}, we review some theoretical background on TV regularization, as well as corresponding reconstruction algorithms. In Section~\ref{sect_TV_Reg}, we then present our modified \ssstar Newton approach to TV regularization, and provide a detailed convergence analysis. In Section~\ref{sect_numerics}, we then apply our approach to two (large-scale) tomographic inverse problems, and conduct extensive numerical experiments, before ending with a short conclusion in Section~\ref{sect_conclusion}.

\section{Background on TV regularization}\label{sect_TV_background}

In this section, we review some background on total variation, TV regularization of linear and nonlinear inverse problems, and corresponding reconstruction algorithms.

\subsection{Background on total variation}

The total variation $\abs{x}_\BV$ of $x \in L^1(\Omega)$, $\Omega \subset \R^n$, was already defined in \eqref{def_BV}. The space
    \begin{equation*}
        \BV := \BV(\Omega) := \Kl{ x \in L^1(\Omega) \, \vert \, \abs{x}_\BV < \infty} \,,
        \qquad
        \text{with}
        \qquad
        \norm{x}_\BV := \abs{v}_{L^1(\Omega)} + \abs{v}_\BV \,,
    \end{equation*}
is a Banach space. The total variation can be seen as the weak* limit \cite{Me} of the Sobolev space $W^{1,1}(\Omega)$, and if $x \in W^{1,1}(\Omega)$, the norms agree. However, unlike $W^{1,1}(\Omega)$, $\BV$ allows for discontinuities, making it highly attractive as a space of ``cartoon'' (or clean) images, as shown in the seminal Rudin-Osher-Fatemi image decomposition model \cite{RuOsFa}.  

Note that since $\phi \in \R^n$, the definition \eqref{def_BV} of $\abs{\cdot}_\BV$ depends on the chosen norm on $\R^n$. The standard Euclidean norm leads to the so-called isotropic $\BV$-functional, which represents the weak* limit of
the norm  
    \begin{equation*}
        \int_\Omega \sqrt{ \sum_{k=1}^n \kl{\frac{\partial}{\partial s_k}x(s) }^2  } \, ds \,.
    \end{equation*}
From an algorithmic point of view, it is often more convenient to use the anisotropic version of BV, which instead of the $\ell_2$-norm uses the $\ell_\infty$-norm on $\R^n$, leading to
    \begin{equation*}
        \int_\Omega \sum_{k=1}^n \abs{\frac{\partial}{\partial s_k}x(s) }  \, ds \,.
    \end{equation*} 
While both functionals define equivalent norms on $\BV$, the limit of a regularization procedure may differ depending on the particular choice among these two options. 

One interesting aspect of the total variation is that it can be computed level-set-wise, which is a consequence of the Coarea formula \cite[Thm.~340]{AmFuPa}: For $x \in BV$, 
    \begin{equation*}
        \abs{x}_\BV = \int_{-\infty}^\infty \abs{\chi_{\Kl{x > t}}}_\BV \, dt \,,
    \end{equation*}
where $\chi$ denotes the indicator function. This connection is useful, e.g., in $L^1(\Omega)$-$\BV$ denoising problems, which decompose level-set-wise \cite{ChEs}, or for the level-set-wise computation of the $\BV$-proximal operator \cite{ChCaCrNoPo}. Another useful analytic property of the total variation is the following approximation property (e.g., \cite[Thm. 3.9]{AmFuPa}): for any $x \in \BV$, there exists a sequence $x_n \in C^\infty(\Omega)$ with 
    \begin{equation*}
        \norm{x_n - x}_{L^1(\Omega)} \to 0 \,,\qquad \text{and} \qquad \abs{x_n}_\BV \to \abs{x}_\BV \,,
    \end{equation*}
and for any such approximation $x_n$, we have $\abs{x}_\BV \leq \liminf_n \abs{x_n}_\BV$. Moreover, the total variation is $L^1(\Omega)$ (or even $L^1_{\mathrm{loc}}(\Omega)$)-lower semicontinuous  \cite[Prop.~3.6]{AmFuPa}, i.e.,
    \begin{equation*}
        x_n \to_{L^1(\Omega)} x \quad\Longrightarrow\quad   \abs{x}_\BV \leq \liminf_{n \to \infty} \abs{x_n}_\BV \,.
    \end{equation*}

A weak topology on $\BV$ is introduced via weak* convergence: a sequence $x_n \in \BV$ converges weak* to $x$, if $x_n \to x$ in $L^1(\Omega)$ and 
    \begin{equation*}
         \int_\Omega  x_n(s) \div \phi(s) \, d s \to \int_\Omega  x(s) \div \phi(s) \, d s \,, 
         \qquad \forall \phi \in C_0^\infty(\Omega;\R^n) \,.
    \end{equation*}
Equivalently \cite[Prop.~3.13]{AmFuPa}, weak* convergence is characterized by $x_n \to x$ in $L^1(\Omega)$ and $\sup_n |x_n|_\BV <\infty$. Also, note that any sequence with $\sup_n \abs{x_n}_\BV < \infty$ has a weak* convergent subsequence \cite[Thm.~3.23]{AmFuPa}. 

The subgradient of the total variation commonly appears in the Euler-Lagrange equations for TV-regularization and, among other things, plays a role in stating source conditions for obtaining convergence rates \cite{BuOs}, as well as in \cite{IgMeSc,ReSc,Scherzer_Grasmair_Grossauer_Haltmeier_Lenzen_2008,HoKaPoSc}. For $\Omega \subset \R^2$, consider the functional $x \mapsto \abs{x}_\BV$ on $L^2(\Omega)$, extended by $+\infty$ outside $\BV$. Then, an element $x^* \in \partial \abs{x}_\BV$ can informally be characterized via the identity \cite[Lemma 5]{Me} 
    \begin{equation*}
        x^*  = -\div \kl{ \frac{\nabla x}{\abs{\nabla x}} } \,,
    \end{equation*}
with substantial technical details at points where $\nabla x = 0$; \cite[Lem.~1]{AlCaCh}, \cite{AnBaCaMa}, \cite[Prop.~8]{BrHo2}. 

The following analytical properties of the $\BV$ space are often useful: sets with bounded $\BV$-norm are sequentially compact in $L^1(\Omega)$ \cite[Thm.~3.23]{AmFuPa}. Moreover, we have a Poincar\'e inequality: for bounded regions $\Omega \subset \R^n$, there holds
    \begin{equation*}
        \abs{x}_{L^p(\Omega)} \leq C \abs{x}_\BV  \,,\qquad \text{for all } x \text{ with } \int_\Omega x(s) \, ds = 0 \,, 
        \qquad 1 \leq p \leq \frac{n}{n-1} \,.
    \end{equation*}
Moreover, the embedding $\BV(\Omega) \hookrightarrow L^p(\Omega)$ is compact for $1 \leq p < \frac{n}{n-1}$ \cite[Cor.~3.49]{AmFuPa}. 

Regarding geometric aspects of total variation, the $\BV$-norm unfortunately does not have favorable properties; for example, the extremal points of the $\BV$ unit ball consist of indicator functions of so-called simple sets; see \cite[Prop.~3.1]{CrIgWa} and \cite[Thm.~4.7]{BrCa}. 

\subsection{Background on TV regularization} 

As noted in the introduction, the total variation $\abs{\cdot}_\BV$ is commonly used as a penalty functional in Tikhonov regularization; see \eqref{TV_regularization}. In the simplest case of $A = I$ with $X=Y=\LtO$, the Tikhonov functional \eqref{TV_regularization} corresponds to the well-known Rudin-Osher-Fatemi denoising functional \cite{RuOsFa}. For general operators $A$, the following assumptions are made to prove existence of a minimizer: (i) $A: L^2(\Omega) \to L^2(\Omega)$ is continuous, and (ii) $A 1 = 1$. According to \cite{ChLi}, a minimizer exists. Condition (ii) is only required due to the fact that constant functions have vanishing $\BV$-norm; it can be removed if, e.g., the norm $\norm{x}_{L^1(\Omega)} + \abs{x}_\BV$ or variants are used as penalty terms instead of $\abs{x}_\BV$. The study of convergence and convergence rates is more involved than in the standard case, and it is useful to consider convergence rates with respect to the Bregman distance 
    \begin{equation*}
        d_p(x,z) := \abs{x}_\BV - \abs{z}_\BV - \spr{ p, x - z } \,,
        \qquad p \in \partial \abs{z}_\BV \,.
    \end{equation*}
Convergence rates for \eqref{TV_regularization} are typically proven under the source condition~\cite{BuOs}
    \begin{equation*}
        \exists \, w \in Y \, : \, A^* w \in \partial \abs{\xD}_\BV \,,
    \end{equation*}
where $\xD$ is the exact solution; see also \cite{Re,Ki} for rates in the reversed Bregman distance. These results have been generalized \cite{HoKaPoSc} to the case of nonlinear ill-posed problems given by an operator equation $F(x) = b$, for which general Tikhonov regularization then reads 
    \begin{equation*}
        \xad := \argmin_{x \in D(F)} \mathcal{T}_\alpha^\delta(x) \,,
        \qquad \text{where} \qquad
        \mathcal{T}_\alpha^\delta(x) :=
        \norm{F(x) - b^\delta}_Y^p + \alpha R(x) \,,
    \end{equation*}
where $R(x)$ is a general convex penalty, e.g., $R(x) = \abs{x}_\BV$. In \cite{HoKaPoSc}, convergence rates in the Bregman distance were established under appropriate assumptions, notably variational inequalities instead of source conditions. This initiated further work on Tikhonov regularization in Banach spaces \cite{ReSc,Scherzer_Grasmair_Grossauer_Haltmeier_Lenzen_2008,T1,T2,T3,T4,T5,T6,T7,T8,T9,Kindermann_Hubmer_2025}, with $\BV$ regularization being a special case. Finally, note that BV penalization can be generalized in various directions, such as to nonlocal $\BV$ \cite{KiOsJo,GiOs} or generalized total variation \cite{BrHo}.

\subsection{Reconstruction algorithms for TV regularization}

In this section, we provide an overview of some of the most important approaches for minimizing the Tikhonov functional \eqref{TV_regularization} with a total variation penalty term. As noted above, especially in earlier approaches, smooth approximations of $\abs{\cdot}_\BV$ such as \eqref{approx} are commonly used to transform \eqref{TV_regularization} into a smooth optimization problem, to which gradient descent methods or iterations based on the Euler–Lagrange equations are then applied \cite{RuOsFa,AcVo}. One example, based on the Barzilei-Borwein method, is the iteration
    \begin{equation}\label{method_BB}
        x_{k+1} = x_k - \tau_k \kl{ A^* (A x_k - b^\delta) - \alpha \nabla \cdot \kl{ \frac{\nabla x_k}{(\abs{\nabla x_k}^2 + \epsilon)^{1/2}}} } \,,
    \end{equation}
where $\epsilon > 0$ is a smoothing parameter, compare with \eqref{approx}, and
    \begin{equation*}
        \tau_k := \frac{\spr{ \Delta x, \Delta x }}{\spr{ \Delta x, \Delta \Tad(x) } } \,,
        \quad
        \Delta x := x_k - x_{k-1}\,,
        \quad
        \Delta T_\alpha^\delta(x) := \nabla \Tad (x_k) -\nabla \Tad (x_{k-1}) \,,
    \end{equation*}
is the Barzilai-Borwein stepsize \cite{Barzilai_Borwein_1988}. The method \eqref{method_BB} is used as the representative of these type of approximate $\abs{\cdot}_\BV$ approaches in the numerical examples considered below. 

On the other hand, when not considering smooth approximations of $\abs{\cdot}_\BV$, one may instead start from the Euler-Lagrange equation \eqref{first_order} for \eqref{TV_regularization}, which can be written as
    \begin{equation*}
        -\tau \kl{A^*(A x - b^\delta) + \alpha \partial \abs{x}_\BV} + x \ni x \,,
    \end{equation*}
where $\tau$ is a stepsize parameter. Moving the subgradient $\alpha \partial \abs{x}_\BV$ to the other side of the inclusion, and applying fixed-point iteration, yields the proximal gradient method 
    \begin{equation}\label{proxgrad}
        x_{k+1} = (I + \alpha \tau \partial \abs{\cdot}_\BV)^{-1} \kl{ x_k - \tau A^*(A x_k - b^\delta)  } \,.
    \end{equation}
Here, the mapping $(I + \alpha \tau \partial \abs{\cdot}_\BV)^{-1}$ is the proximal mapping $\prox_{\alpha \tau \abs{\cdot}_\BV}$ of the scaled total variation functional $\alpha \tau \abs{\cdot}_\BV$, which can equivalently be defined as 
    \begin{equation*}
        \prox_{\alpha \tau \abs{\cdot}_\BV}(x)
        := 
        \arg\min_{z} J_{\alpha \tau}(z) \,,
        \qquad
        \text{where}
        \qquad
        J_\lambda(z) := \frac{1}{2} \norm{z-x}_X^2 + \lambda \abs{z}_\BV\,.
    \end{equation*}
Although it is well-known that $\prox_{\alpha \tau \abs{\cdot}_\BV}(x)$ is single-valued \cite{Be17}, its computation is non-trivial. Note that the proximal functional $J(x)$ coincides with the Tikhonov functional \eqref{TV_regularization} for the denoising problem, i.e., $A = I$. The functional $J(x)$ has a dual of the form
    \begin{equation}\label{eq:dual}
        J^*(w) = \frac{1}{2} \norm{w - \frac{g}{\lambda}}_{L^2(\Omega)}^2 + K^*(w) \,,
    \end{equation}
where $K^*$ is the characteristic function of the set $\Kl{ \div \xi \mid |\xi| \leq 1 }$. The minimizer of $J^*(w)$ is given by $w = g - P_{\lambda K} g$, where $P_{\lambda K}$ denotes the projector onto the set $\lambda K$, which can be obtained by solving
    \begin{equation}\label{eq:dual2}
        \min_p \norm{ \lambda \op{div} p - q}^2 \,, \qquad \text{subject to } \abs{p} \leq 1 \,.
    \end{equation}
This problem can be solved by a semi-implicit gradient descent method without any operator inversion, involving only simple $\ell_1$-type projections, leading to the so-called dual method of Chambolle \cite{Ch}. 
A competitive improvement of this method, proposed by Beck and Teboulle \cite{BeTe}, leads to the well-known FISTA algorithm, which employs Nesterov acceleration of the gradient descent and requires only a simple additional extrapolation step. Finally, regarding the Tikhonov functional \eqref{TV_regularization}, FISTA can be used in an inner iteration for the proximal map together with the outer iteration \eqref{proxgrad} \cite{BeTe2}.

The various proximal gradient-type methods were subsequently improved in \cite{ChPo} by employing a primal–dual algorithm, now referred to as the Chambolle–Pock method. For the TV regularization problem \eqref{TV_regularization}, the method (with parameters $\sigma,\tau$) reads
    \begin{equation}\label{Chambolle_Pock}
    \begin{split}
        p_{n+1} &= \kl{I + \sigma \partial 
       \delta_{\B_{\infty}}}^{-1}\kl{ p_n + \sigma \nabla \bar{x}_n } \,, 
        \\
        x_{n+1} &= (I + \tau A^*A)^{-1} \kl{ x_n + \tau A^* b^\delta- \tau \op{div} p_{n+1} } \,, 
        \\
        \bar{x}_{n+1} &= x_{n+1} + \theta (x_{n+1} - x_n) \,,
    \end{split}
    \end{equation}
where $\delta_{\B_{\infty}}$ denotes the indicator function of the unit ball of vector-valued $L^\infty(\Omega)$-functions. In the first line, the proximal map is essentially identical to the iteration step in Chambolle's dual method, while the second line is a primal step using the forward operator. The last line corresponds to an extrapolation step.

A common alternative to the Chambolle–Pock method is the alternating direction method of multipliers (ADMM) method \cite{BoydADMM}, which is defined by
    \begin{equation*}
    \begin{split}
        x_{n+1} &= (A^*A + \Delta)^{-1} \kl{A^*b^\delta + \tau \op{div}(z_n - u_n)} \,, 
        \\
        z_{n+1} &= \kl{I + \frac{\lambda}{\tau} \partial \norm{\cdot}_{L^1}}^{-1} (\nabla x_{n+1} + u_n) \,, 
        \\
        u_{n+1} &= u_n + \nabla x_{n+1} - z_{n+1} \,.
    \end{split} 
    \end{equation*}
Note that it was later found that an essentially equivalent form of the ADMM method was proposed by Goldstein and Osher under the name of Split-Bregman iteration \cite{GoOs}. A slight difference in that method is that Gauss-Seidel iterations are used in the first line of the iteration, but otherwise the structure of the method is equivalent. 

Concerning second-order methods, a Newton iteration for the primal-dual optimality equations has been proposed by Chan et.al. in \cite{ChanGolub}. A more advanced approach is the semismooth Newton method proposed by Hintermüller and Kunisch \cite{HiKu}. This method is not directly applicable to \eqref{TV_regularization} when $A^*A$ is not invertible, and thus an additional regularization penalty $\gamma \norm{x}_{L^2(\Omega)}^2$ is added. In this case, there is a dual functional:
    \begin{equation}\label{eq:dual3}
        \min_p \norm{(A^*A + \alpha I)^{-\frac{1}{2}} (\lambda \op{div} p + A^* b^\delta)}^2 \,, \qquad \text{subject to } |p| \leq 1 \,.
    \end{equation}
from whose minimizer $p$ the solution $x$ can be recovered via $\op{div} p = (A^*A +\alpha I)^{\frac{1}{2}} x - A^*y$. Formally, and on a discretized level, one can apply a semismooth Newton method to the dual functional as follows: First, one considers Lagrange multipliers for the constraints on $p$, which yield the optimality conditions 
    \begin{equation*}
    \begin{split}
        &\lambda \nabla (A^*A +\alpha I)^{-1}(\lambda \op{div} p + A^* b^\delta) +\mu = 0 \,,
        \\ 
        &\mu = \max\Kl{0,\mu + c(p -1)} +  \min\Kl{0,\mu + c(p +1)} \,,
    \end{split} 
    \end{equation*}
where $c>0$ is a fixed parameter/stepsize. 
Note that in \cite{HiKu}, an additional 
regularization term is added to the 
dual problem. 
A semi-smooth Newton method for this problem then consists of an active set strategy: at iteration $k$, let $\mathcal{A}_{k+1}^+$ and $\mathcal{A}_{k+1}^-$ denote the sets of indices where $\mu_k + c(p_k -1) >0$ and $\mu_k + c(p_k +1) < 0$, respectively, and let $\mathcal{I}_{k+1}$ denote the complementary index set. Then, $p_{k+1}$ and $\mu_{k+1}$ are defined by 
    \begin{equation*}  
    \begin{split}
        &\mu_{k+1} = 0 \,, \qquad \text{for} \quad i \in \mathcal{I}_{k+1} \,,
        \\ 
        &(p_{k+1})_i  = 1\,, \qquad i \in \mathcal{A}_{k+1}^+ \,, \quad (p_{k+1})_i  = -1\,, \quad  i \in \mathcal{A}_{k+1}^-
        \\ 
         &\lambda \nabla (A^*A +\alpha I)^{-1}\kl{\lambda\op{div} p_{k+1} + A^* b^\delta} + \mu_{k+1} = 0 \,.
    \end{split}
    \end{equation*} 
Note that the last equation can be solved by splitting the index sets into $\mathcal{I}_{k+1}$ and its complement, such that then only a linear system of size $\abs{\mathcal{I}_{k+1}} \times \abs{\mathcal{I}_{k+1}}$ has to be solved.

\section{\texorpdfstring{A \ssstar Newton augmented Lagrangian approach for the efficient solution of $\ell_{1}$-regularized least-squares problem}{A semismooth* Newton approach for the efficient solution of l1-regularized least-squares problem}
}\label{sect_TV_Reg}

As noted above, after discretization, the TV regularization problem \eqref{TV_regularization} typically turns into a finite-dimensional optimization problem of the form \eqref{TV_reg_discr}. Hence, in this section, we consider the efficient solution of $\ell_1$-regularized least-squares problems of the form
    \begin{equation}\label{EqLSProbl}
        \min_x \varphi(x) := \frac 12 \norm{Ax-b^\delta}^2 +\alpha \norm{Bx}_1 \,,
    \end{equation}
where $A$ and $B$ are $m \times n$ and $l \times n$ matrices, respectively, $b^\delta \in \R^m$, and $\alpha > 0$. For the special case \eqref{TV_reg_old} with $A$ arising from applications in tomography, this problem has the following features complicating the design of an efficient solution algorithm:
\begin{itemize}
    \item Both matrices $A$ and $B$ may be very large, and possibly $A$ is not known explicitly but only via a routine for evaluating the matrix-vector products $Ax$ and $A^T b$. As a consequence, linear systems involving $A$ can be solved only iteratively.
    \item The objective $\varphi$ is not strongly convex, i.e., $A^TA$ is singular. This prevents the use of some dual methods (such as \cite{HiKu}) without adding some regularization term.
    \item The proximal mapping of the nonsmooth function $x\mapsto\alpha\norm{Bx}_1$ is not easily accessible. This implies that a number of well-known methods from nonsmooth composite optimization, e.g., \cite{Gfr25a,LiSuTo18, StThPa17}, cannot be applied.
    \item We have to expect that the problem is dually degenerated in the sense that the subgradient $z^*$ appearing in the first-order optimality conditions \eqref{EqFO_ExtProbl} below is not unique. \end{itemize}

\subsection{Existence of solutions/minimizers}

First, we show that our minimization problem \eqref{EqLSProbl} has a non-empty solution set.

\begin{proposition}\label{PropExSolution}
The solution set $S_{\rm Opt}:=\argmin \varphi(x)$ is not empty.
\end{proposition}
\begin{proof}
Let $L:=\ker A\cap\ker B$. Since $\varphi$ is bounded from below by $0$, there exists a minimizing sequence $x_k$ with $\lim_{k\to\infty}\varphi(x_k)=\inf\varphi$. Since for every $u\in L$ we have $\varphi(x_k+u)=\varphi(x_k)$, we may assume that $x_k\in L^\perp$ for all $k$. We now show by contraposition that the sequence $x_k$ is bounded. For this, assume on the contrary that $x_k$ has an unbounded subsequence, without loss of generality the sequence $x_k$ itself. By possibly passing to a subsequence, we may assume that $u_k:=x_k/\norm{x_k}$ converges to some $u$ and, by taking into account that the sequence $\norm{Ax_k-b^\delta}^2$ is bounded, we obtain that
    \begin{equation*}
        0=\lim_{k\to\infty}\frac{\varphi(x_k)}{\norm{x_k}}=\lim_{k\to\infty}\Big(\frac{\norm{Ax_k-b^\delta}^2}{2\norm{x_k}}+\alpha\norm{Bu_k}_1\Big)=\alpha\norm{Bu}_1 \,,
    \end{equation*}
and
    \begin{equation*}
        0=\lim_{k\to\infty}\frac{\varphi(x_k)}{\norm{x_k}^2}=\lim_{k\to\infty}\Big(\frac 12\norm{Au_k-\frac{b^\delta}{{\norm{x_k}}}}^2+\alpha\frac{\norm{Bu_k}_1}{\norm{x_k}}\Big)=\frac 12\norm{Au}^2 \,.
    \end{equation*}
Hence $u\in L$, $\norm{u}=1$, and for all $k$ sufficiently large we have $\frac 12 <\skalp{u,u_k}=\skalp{u,x_k}/\norm{x_k}$ contradicting $x_k\in L^\perp$. Hence, $x_k$ is bounded and, after possibly passing to a subsequence, we may assume that $x_k$ converges to some $\xb$. By continuity of $\varphi$ we readily obtain $\varphi(\xb)=\inf\varphi$, proving the assertion.
\end{proof}

Using the above result, we conclude from \cite[Theorem 23.9]{Ro70} that at every solution $x\in S_{\rm Opt}$ the following first-order optimality condition is fulfilled: 
    \begin{equation}\label{EqFO_Probl}
        0\in\partial\varphi(x)=A^T(Ax-b^\delta)+ B^T\partial\alpha\norm{Bx}_1 \,.
    \end{equation}

\subsection{On an augmented Lagrangian method}

Next, note that by substituting $z:=Bx$, problem \eqref{EqLSProbl} can be equivalently written as
    \begin{equation}\label{EqExtProblem}
        \min_{x,z} \frac 12 \norm{Ax-b^\delta}^2 +\alpha \norm{z}_1\,,
        \quad \text{subject to} \quad Bx-z=0 \,,
    \end{equation}
and the first-order optimality conditions for this problem read
    \begin{subequations}\label{EqFO_ExtProbl}
    \begin{align}
        \label{EqFO1a}  &A^T(Ax-b^\delta)+ B^Tz^*=0 \,,
        \\
        \label{EqFO1b}  &Bx-z=0 \,, 
        \\
        \label{EqFO1c}  &z^*\in\partial\alpha\norm{z}_1 \,.
    \end{align}
    \end{subequations}
For every solution $(x,z)$ of problem \eqref{EqExtProblem}, there exists some multiplier $z^*$ such that the triple $(x,z,z^*)$ fulfills \eqref{EqFO_ExtProbl}. Conversely, if we are given a triple $(x,z,z^*)$ fulfilling the first-order optimality conditions \eqref{EqFO_ExtProbl}, then $(x,z)$ is a solution of the convex program \eqref{EqExtProblem}. Our goal is now to compute, with some numerical procedure, a triple $(x,z,z^*)$ such that $z^*\in\alpha\partial\norm{z}_1$ and both $\norm{A^T(Ax-b^\delta)+B^Tz^*}$ and $\norm{Bx-z}$ are small. By the following statement, we can then conclude that $(x,z)$ is close to a solution of \eqref{EqExtProblem}.

\begin{proposition}\label{PropErrEstimate}
The set
    \begin{equation*}
        S_{\rm FO}:=\Kl{(x,z,z^*)\mv (x,z,z^*) \,\, \text{fulfills} \,\, \eqref{EqFO_ExtProbl}}
    \end{equation*}
is non-empty, and there exist $0 < \kappa, \bar\eps \in \R$ such that for every $(x,z,z^*)\in\R^n\times\R^l\times\R^l$ satisfying $\norm{A^T(Ax-b^\delta)+ B^Tz^*} +\norm{Bx-z}+\dist{z^*,\partial\alpha\norm{z}_1}<\bar\eps$, there holds
    \begin{equation}\label{EqErrEstimate}
    \begin{split}
        \dist{x,S_{\rm Opt}}&\leq \dist{(x,z,z^*),S_{\rm FO}}
        \\
        &\leq \kappa\big(\norm{A^T(Ax-b^\delta)+ B^Tz^*}+\norm{Bx-z}+\dist{z^*,\partial\alpha\norm{z}_1}\big) \,.
    \end{split}
    \end{equation}
\end{proposition}
\begin{proof}
Consider $\xb\in S_{\rm Opt}$ and a subgradient $\zba\in\partial \alpha\norm{B\xb}_1$ fulfilling the first-order optimality condition \eqref{EqFO_Probl}, i.e., $0=A^T(A\xb-b^\delta)+ B^T\zba$. Then we obviously have $(\xb,B\xb,\zba)\in S_{\rm FO}\not=\emptyset$. Next, consider the set-valued mapping
    \begin{equation*}
        F:\R^n\times\R^l\times\R^l\tto\R^n\times\R^l\times\R^l \,,
        \qquad 
        F(x,z,z^*):=\myvec{A^T(Ax-b^\delta)+ B^Tz^*\\Bx-z\\z^*-\partial \alpha \norm{z}_1} \,.
    \end{equation*}
Since $\partial\alpha\norm{\cdot}_1$ is a polyhedral multifunction, i.e., its graph is the union of finitely many convex polyhedra, so is $F$. Hence, by \cite{Rob81} there are $0 < \kappa, \bar\eps \in \R$ and such that
    \begin{equation*}
        \dist{(x,z,z^*), F^{-1}(0)}\leq \kappa \,\dist{0, F(x,z,z^*)}\,, \quad \text{whenever} \quad \dist{0, F(x,z,z^*)}<\bar\eps \,,
    \end{equation*}
from which the bound \eqref{EqErrEstimate} directly follows.
\end{proof}

Next, for a given penalty parameter $\sigma\geq0$, consider the augmented Lagrangian
    \begin{equation*}
    \begin{split}
        \lag_\sigma:\R^n\times\R^l\times\R^l &\to\R
        \\
        \lag_\sigma(x,z,\zeta^*) &:= \frac 12 \norm{Ax-b^\delta}^2 +\alpha \norm{z}_1+\skalp{\zeta^*,Bx-z}+\frac\sigma2\norm{Bx-z}^2 \,.
    \end{split}
    \end{equation*}
Now as a numerical method for solving \eqref{EqExtProblem}, we propose the following inexact variant of an \emph{augmented Lagrangian method (ALM)}:

\begin{algorithm}[Algorithm~ALM]\label{AlgALM}\mbox{ }\\
Let $\beta\in(0,1)$, a sequence $\gamma_l > 0$ satisfying $\sum_{l=0}^\infty\gamma_l=\infty$, a penalty parameter $\ee\sigma0>0$, and a starting point $(\ee x0,\ee z0,\ee{\zeta^*}0)\in \R^n\times\R^l\times\R^l$ be given. Set $\ee l0:=0$.
\\ \vspace{-5pt} \\ \noindent 
For $k=0,1,\ldots$, perform the following steps
\begin{itemize}
    \item {\bf Step 1:} Compute
\begin{equation}
  \label{EqApprMinLaq}(\ee x{k+1},\ee z{k+1})\approx\argmin_{x,z}\lag_{\ee\sigma k}(x,z,\ee {\zeta^*}k) \,.
\end{equation}
\item {\bf Step 2:} Set
\begin{equation}
  \label{EqUpdateMult}\ee{\zeta^*}{k+1}:=\ee{\zeta^*}k+\ee\sigma k(B\ee x{k+1}-\ee z{k+1}) \,,
\end{equation}
and
\begin{equation}\label{EqUpdateSigma}
\begin{split}
  &\big(\ee \sigma{k+1},\ee l{k+1}\big)
  \\
  &:=\begin{cases}
  \big((1+\gamma_{\ee lk})\ee\sigma k,\ee lk +1\big) \,, &\mbox{if $\norm{B\ee x{k+1}-\ee z{k+1}}>\beta\norm{B\ee xk-\ee zk}$} \,, \\
  \big(\ee\sigma k,\ee lk\big) \,, &\mbox{otherwise} \,.
\end{cases}
\end{split}
\end{equation}
\end{itemize}
\end{algorithm}

In order to ensure convergence of Algorithm~\ref{AlgALM}, we have to specify the level of accuracy in \eqref{EqApprMinLaq}. Note that for fixed $x$, the minimization of $\lag_\sigma(x,z,z^*)$ wrt.\ $z$ can be easily performed: For given $\tau>0$, consider the {\em Moreau envelope}  of $\tau\norm{\cdot}_1$, defined by
    \begin{equation}\label{EqMoreau}
        \Mor \tau (\zeta):= \min_{z}\frac 12\norm{z-\zeta}^2+\tau\norm{z}_1 \,,
    \end{equation}
and the \emph{proximal mapping}
    \begin{equation}\label{EqProx}
        \Prox \tau(\zeta):=\argmin_{z}\frac 12\norm{z-\zeta}^2+\tau\norm{z}_1 \,,
    \end{equation}
i.e., the $i$-th component is \cite{Be17}
    \begin{equation}\label{EqProxL1}
        (\Prox \tau(\zeta))_i 
        = \begin{cases}
        \zeta_i-\tau \,, &\mbox{if $\zeta_i>\tau$} \,, 
        \\
        0 \,, & \mbox{if $\zeta_i\in[-\tau,\tau]$} \,, 
        \\
        \zeta_i+\tau \,, &\mbox{if $\zeta_i<-\tau$} \,.
        \end{cases}
    \end{equation}
$\Prox\tau$ is also called \emph{soft-thresholding operator}. It is straightforward to verify that
    \begin{align}
        \nonumber\min_z\lag_\sigma(x,z,\zeta^*)&=\frac 12 \norm{Ax-b^\delta}^2-\frac {\norm{\zeta^*}^2}{2\sigma}+\sigma\min_z\Big(\frac 12\norm{z-(Bx+\frac{\zeta^*}\sigma)}^2+\frac \alpha\sigma \norm{z}_1\Big)
        \\
        \label{EqMin_z_Lag}&=\frac 12 \norm{Ax-b^\delta}^2-\frac {\norm{\zeta^*}^2}{2\sigma}+\sigma \Mor{\alpha/\sigma}\Big(Bx+\frac{\zeta^*}\sigma\Big) \,,
    \end{align}
and therefore
    \begin{equation}\label{EqArgMinLag}
        \argmin_z\lag_\sigma(x,z,\zeta^*)=\Prox{\alpha/\sigma}\Big(Bx+\frac {\zeta^*}\sigma\Big) \,.
    \end{equation}
Thus, it is evident to require that
    \begin{equation}\label{Eq_z_k+1}
        \ee z{k+1}=\Prox{\alpha/\ee\sigma k}\Big(B\ee x{k+1}+\frac{\ee{\zeta^*}k}{\ee \sigma k}\Big)\,, \qquad \forall \, k \,,
    \end{equation}
implying by virtue of the first-order optimality conditions that
    \begin{equation*}
        0\in\partial_z\lag_{\ee\sigma k}(\ee x{k+1},\ee z{k+1},\ee{\zeta^*}k)=- \ee {\zeta^*}k-\ee\sigma k(B\ee x{k+1}-\ee z{k+1})+\partial \alpha \norm{\ee z{k+1}}_1 \,,
    \end{equation*}
and consequently, by \eqref{EqUpdateMult}, 
    \begin{equation}\label{EqFO_z}
        \ee {\zeta^*}{k+1}\in \partial \alpha \norm{\ee z{k+1}}_1 \,.
    \end{equation}
With these considerations, we now obtain the following convergence result:

\begin{theorem}
Let $(\ee xk,\ee zk,\ee{\zeta^*}k)$ be produced by Algorithm~\ref{AlgALM}. If \eqref{Eq_z_k+1} holds and
    \begin{equation}\label{EqNabla_x_lag}
    \begin{split}
        &\lim_{k\to\infty}\norm{\nabla_x\lag_{\ee\sigma k}(\ee x{k+1},\ee z{k+1},\ee{\zeta^*}k)}\\
        &=\lim_{k\to\infty}\norm{A^T(A\ee x{k+1}-b^\delta)+B^T\kl{\ee {\zeta^*}k+\ee \sigma k\kl{B\ee x{k+1}-\ee z{k+1}}}}=0 \,,
    \end{split}
    \end{equation}
then
    \begin{equation*}
        \lim_{k\to\infty}\dist{\kl{\ee xk,\ee zk,\ee {\zeta^*}k},S_{\rm FO}}=0 \,,
    \end{equation*}
and
    \begin{equation*}
        \lim_{k\to\infty}\dist{\ee xk,S_{\rm Opt}}=0 \,.
    \end{equation*}
In particular, every accumulation point of $\ee xk$ is a solution of problem \eqref{EqLSProbl}.
\end{theorem}
\begin{proof}
Since the subdifferential of the $\ell_1$-norm is contained in the unit ball with respect to the $\ell_\infty$-norm, we conclude from \eqref{EqFO_z} that $\Vert \ee {\zeta^*}{k+1} \Vert_\infty\leq \alpha$ for all $k$. If $\ee \sigma{k+1}$ is only finitely many times increased by the update scheme \eqref{EqUpdateSigma}, then we have 
    \begin{equation*}
        \norm{B\ee x{k+1}-\ee z{k+1}}\leq\beta\norm{B\ee xk-\ee zk}
        \qquad \text{for all $k$ sufficiently large} \,,
    \end{equation*}
implying $\lim_{k\to\infty}\norm{B\ee xk-\ee zk}=0$. On the other hand, if $\ee \sigma{k+1}$ is increased infinitely many times, we conclude from $\sum_{k=0}^\infty\gamma_l=\infty$ that $\lim_{k\to\infty}\ee \sigma k=\infty$, and therefore,
    \begin{equation*}
        \lim_{k\to\infty}\norm{B\ee x{k+1}-\ee z{k+1}}=\lim_{k\to\infty}\frac{\norm{\ee {\zeta^*}{k+1}-\ee{\zeta^*}k}}{\ee \sigma k}=0 \,,
    \end{equation*}
where we have used \eqref{EqUpdateSigma} and the boundedness of $\ee {\zeta^*} k$. Hence, together with \eqref{EqUpdateMult} and \eqref{EqFO_z}, we obtain from \eqref{EqNabla_x_lag} that
    \begin{align*}
        &\lim_{k\to\infty}\norm{A^T(A\ee x{k+1}-b^\delta)+B^T\ee {\zeta^*}{k+1}}+\norm{B\ee x{k+1}-\ee z{k+1}}+\dist{\ee {\zeta^*}{k+1},\alpha \partial\norm{\ee z{k+1}}_1}
        \\
        &= \lim_{k\to\infty}\norm{A^T(A\ee x{k+1}-b^\delta)+B^T\big(\ee {\zeta^*}k+\ee \sigma k(B\ee x{k+1}-\ee z{k+1})\big)}+\norm{B\ee x{k+1}-\ee z{k+1}}\\
        &=0 \,,
    \end{align*}
and thus the assertion now follows from Proposition~\ref{PropErrEstimate}.
\end{proof}

\subsection{On the efficient solution of the subproblems}

In this section, we consider a regularized \ssstar Newton method for the approximate solution of \eqref{EqApprMinLaq} in Algorithm~\ref{AlgALM}. Omitting, for easier readability, the iteration index $k$, we want to approximately minimize the function $\psi:\R^n\times\R^l\to\R$ given by
    \begin{equation*}
        \psi(x,z):=\frac 12 \norm{Ax-b^\delta}^2+\alpha\norm{z}_1+\skalp{\zeta^*,Bx-z}+\frac \sigma 2\norm{Bx-z}^2= \lag_\sigma(x,z,\zeta^*) \,,
    \end{equation*}
where $\zeta^*\in\R^l$ and $\alpha,\sigma>0$ are fixed parameters. By convexity of $\psi$, this is equivalent to solving the inclusion
    \begin{equation}\label{EqInclPsi}
        0\in\partial\psi(x,z) \,.
    \end{equation}
Now due to \eqref{EqArgMinLag} and \eqref{EqProxL1}, we have
    \begin{equation}\label{EqPsi}
        \Psi(x):=\argmin_z\psi(x,z)=\Prox{\alpha/\sigma}\kl{Bx+\frac{\zeta^*}\sigma} \,,
    \end{equation}
with
    \begin{equation}\label{EqPsi_i}
        \Psi_i(x)=
        \begin{cases}
        (Bx)_i+\frac{\zeta^*_i}\sigma-\frac\alpha\sigma \,, & \mbox{if $(Bx)_i+\frac{\zeta^*_i}\sigma>\frac\alpha\sigma$} \,,
        \\
        (Bx)_i+\frac{\zeta^*_i}\sigma+\frac\alpha\sigma \,, & \mbox{if $(Bx)_i+\frac{\zeta^*_i}\sigma<-\frac\alpha\sigma$} \,,
        \\
        0 \,, &\mbox{otherwise} \,.
        \end{cases}
    \end{equation}
In view of \eqref{EqNabla_x_lag}, we want to find $\tilde x\in\R^n$ such that
    \begin{equation}\label{EqSubProbl}
        \norm{\nabla_x \psi(\tilde x,\Psi(\tilde x))}\leq \eps
    \end{equation}
for some prescribed tolerance $\eps>0$. Note that the choice $z=\Psi(x)$ ensures that
    \begin{equation*}
        0\in\partial_z\psi(x,\Psi(x)) \,.
    \end{equation*}

Now it is well-known (see, e.g., \cite[Proposition 13.37]{RoWe98}) that the Moreau envelope $\Mor\tau(\zeta)$ defined in \eqref{EqMoreau} is continuously differentiable with Lipschitzian gradient
    \begin{equation*}
        \nabla\Mor\tau(\zeta)=\zeta-\Prox \tau(\zeta) \,.
    \end{equation*}
In view of \eqref{EqMin_z_Lag}, it thus follows that the function $\vartheta:\R^n\to\R$ defined by
    \begin{equation}\label{EqTheta}
        \vartheta(x):=\min_{z}\psi(x,z)=\psi\kl{(x,\Psi(x)}
    \end{equation}
is continuously differentiable with Lipschitzian gradient
    \begin{equation}\label{EqGradTheta}
        \nabla\vartheta(x)= A^T(Ax-b^\delta)+B^T\big(\zeta^*+\sigma(Bx-\Psi(x))\big)=\nabla_x\psi(x,\Psi(x)) \,.
    \end{equation}
Since the Moreau envelope of a convex function is again a convex function, we infer from \eqref{EqMin_z_Lag} that $\vartheta$ is convex. It follows that $(\xb,\zb)$ minimizes $\psi$ if and only if $\zb=\Psi(\xb)$ and $\xb$ minimizes $\vartheta$.

For finding approximate zeros of $\nabla\vartheta$, semismooth Newton methods are well established, cf. \cite{QiSun93}. However, these methods usually do not take into account the specific structure of the underlying problem and for this reason, we now use a globalized version of the SCD \ssstar Newton method introduced in \cite{GfrOut22}, which offers more flexibility. Let us briefly recall this method, which aims to solve inclusions of the form
    \begin{equation}\label{EqIncl}
        0\in F(u) \,,
    \end{equation}
where $F:\R^n\tto\R^n$ is a set-valued mapping. For this, consider the metric space $\Z_n$ of all $n$-dimensional subspaces of $\R^n\times \R^n$ equipped with the metric
    \begin{equation*}
        d_\Z(L_1,L_2)=\norm{P_{L_1}-P_{L_2}} \,,
    \end{equation*}
where $P_{L_i}$, denotes the orthogonal projector onto $L_i$. Given a subspace $L\in  \Z_n$, let
    \begin{equation*}
        L^*:=\{(v^*,u^*)\in\R^n\times\R^n\mv (u^*,-v^*)\in L^\perp\} \,.
    \end{equation*}
denote its \emph{adjoint} subspace. Then $(L^*)^*=L$ and $d_\Z(L_1,L_2)=d_\Z(L_1^*,L_2^*)$, cf.~\cite{GfrOut22}. Since $\dim L^*=\dim L^\perp = n+n-\dim L=n$, we have $L^*\in\Z_{n}$ whenever $L\in\Z_{n}$.

\begin{definition}[{\cite[Definition 3.3]{GfrOut22}}]\label{DefSCD}
Let $F:\R^n\tto\R^n$ be a mapping.
    \begin{enumerate}
        \item $F$ is called \emph{graphically smooth} at $(u,v)\in\gph F$ and of dimension $d$ in this respect, if $T_{\gph F}(u,v)$ is a $d$-dimensional subspace of $\R^n\times\R^n$. We denote by $\OO_F$ the set of all points from the graph of $F$, where $F$ is graphically smooth of dimension $n$.
        \item The \emph{subspace containing derivative} (SCD), $\Sp F:\gph F\tto  \Z_n$, is defined by
            \begin{equation*}
                \Sp F(u,v):=\Kl{L\in \Z_{n}\mv \exists (u_k,v_k)\longsetto {\OO_F}(u,v): d_\Z(T_{\gph F}(u_k,v_k),L)=0} \,,
            \end{equation*}
        and the \emph{adjoint} SCD, $\Sp^*F:\gph F\tto \Z_n$, is defined by
            \begin{equation*}
                \Sp^*F(u,v):=\Kl{L^*\mv L\in\Sp F(u,v)} \,.
            \end{equation*}
        \item We say that $F$ has the \emph{SCD property} at $(\ub,\vb)\in\gph F$, if $\Sp F(\ub,\vb)\not=\emptyset$, and we say that $F$ is an  SCD mapping if it has the SCD property at every $(u,v)\in\gph F$.
    \end{enumerate}
\end{definition}

\begin{definition}[{\cite[Definition 5.1]{GfrOut22}}]\label{DefSCDssstar}
We say that $F:\R^n\tto\R^n$ is {\em\SCD \ssstar at} $(\ub,\vb)\in\gph F$, if $F$ has the \SCD property around $(\ub,\vb)$ and for every $\eps>0$ there is some $\delta>0$ such that
    \begin{align}\label{EqDefSSCSemiSmooth}
        \vert \skalp{u^*,u-\ub}-\skalp{v^*,v-\vb}\vert&\leq \eps
        \norm{(u,v)-(\ub,\vb)}\norm{(u^*,v^*)}
    \end{align}
holds for all $(u,v)\in \gph F\cap \B_\delta(\ub,\vb)$ and all $(v^*,u^*)\in L^*$, $L^*\in\Sp^*F(u,v)$.
\end{definition}

Let us now describe one iteration step of the \SCD \ssstar Newton method introduced in \cite{GfrOut21} for solving \eqref{EqIncl}. Assume we are given some iterate $\ee uj$. Since we cannot expect in general that $F(\ee uj)\not=\emptyset$ or that $0$ is close to $F(\ee uj)$, even if $\ee uj$ is close to a solution $\ub$, we first perform a preparatory step, the so-called {\em approximation step}, which yields $(\ee{\hat u}j,\ee{\hat v}j)\in\gph F$ as  an approximate projection of $(\ee uj,0)$ onto $\gph F$. For this approximation step, we require that 
    \begin{equation*}
        \norm{(\ee{\hat u}j,\ee{\hat v}j)-(\ub,0)}\leq \eta\norm{\ee uj-\ub} \,,
    \end{equation*}
for some constant $\eta>0$. Then, in the so-called \emph{Newton step}, we compute two $n\times n$ matrices $\ee Uj$ and $\ee Vj$ such that
    \begin{equation*}
        \rge({\ee Vj}^T, {\ee Uj}^T):=\Kl{({\ee Vj}^Tu, {\ee Uj}^Tu)\mv u\in\R^n}\in\Sp^*F(\ee{\hat u}j,\ee{\hat v}j) \,,
    \end{equation*}
determine the Newton direction $\ee{\triangle u}j$ as solution of the linear system
    \begin{equation}\label{EqNewtonStep}
        \ee Uj\triangle u=-\ee Vj\ee{\hat v}j
    \end{equation}
and set the next iterate as $\ee u{j+1}=\ee {\hat u}j+\ee{\triangle u}j$. The following convergence result holds:

\begin{theorem}[cf. {\cite[Corollary 5.6]{GfrOut22}}]\label{ThConvNewton}
Let $\ub$ be a solution of the inclusion \eqref{EqIncl}, assume that $F$ is \SCD \ssstar at $(\ub,0)\in\gph F$, and assume that for every subspace $L^*\in\Sp F(\ub,0)$ the regularity condition holds:
    \begin{equation*}
        (v^*,0)\in L^*\ \Longrightarrow v^*=0 \,.
    \end{equation*}
Then for every starting point $\ee u0$ sufficiently close to $\ub$, the procedure described above either stops after finitely many steps at a solution of \eqref{EqIncl} or produces a sequence $\ee uj$ which converges superlinearly to $\ub$.
\end{theorem}

Note that Theorem~\ref{ThConvNewton} guarantees convergence only for starting points sufficiently close to a solution. Hence, in the following, we describe a globalized variant of the SCD \ssstar Newton method for finding a zero of $\partial\psi$ which ensures convergence from arbitrary starting points:

Given an iterate $(\ee xj,\ee zj)$, we compute in the approximation step
    \begin{equation*}
        \ee {\hat x}j\approx \argmin_x\psi(x,\ee zj)\,, \qquad \ee{\hat z}j=\Psi(\ee{\hat x}j) \,,
    \end{equation*}
resulting in
    \begin{equation*}
        (\ee{\hat x{}^*}j, \ee{\hat z{}^*}j)=\big(A^T(A\ee{\hat x}j-b^\delta)+B^T(\zeta^*+\sigma(B\ee{\hat x}j-\ee{\hat z}j),0)\big)\in\partial\psi(\ee{\hat x}j,\ee{\hat z}j)\,.
    \end{equation*}
Note that we have at our disposal also another subgradient, namely
    \begin{equation*}
        \ee {z^*}j= \zeta^*+\sigma(B\ee{\hat x}j-\ee{\hat z}j)\in \partial\alpha\norm{\ee{\hat z}j}_1 \,,
    \end{equation*}
which satisfies
    \begin{equation*}
        \ee{z_i^*}j
        \begin{cases}
        =\alpha \,, &\mbox{if $\ee{\hat z_i}j>0$} \,, \\
        \in[-\alpha,\alpha] \,, &\mbox{if $\ee{\hat z_i}j=0$} \,, \\
        =-\alpha \,, &\mbox{if $\ee{\hat z_i}j<0$} \,.
        \end{cases}
    \end{equation*}
Now we analyze the Newton step. The SCD of $\partial \norm{\cdot}_1$ is well-known, see, e.g., \cite[Example 3.8]{Gfr25a} and analogously, one can derive $\Sp(\partial\alpha\norm{\cdot}_1)$. We find that  $\Sp(\partial \alpha\norm{\cdot}_1)(\ee{\hat z}j,\ee{z^*}j)$ consists of the collection of all subspaces $\rge(P,W)$ with diagonal matrices $P$ and $W$ satisfying
    \begin{equation}\label{EqP}
        P_{ii}=
        \begin{cases}
        1 \,, & \mbox{if $i\in I^+\cup I^-$} \,, \\
        0 \,, & \mbox{else,}\end{cases}
        \qquad \text{and} \qquad
        W=I-P \,,
    \end{equation}
where $I^+,I^-$ are index sets satisfying
    \begin{equation*}
        \{i\mv \ee{\hat z_i}j>0\}\subset I^+\subset\{i\mv \ee{z_i^*}j=\alpha\}\,, \quad \text{and} \quad \{i\mv \ee{\hat z_i}j<0\}\subset I^-\subset\{i\mv \ee{z_i^*}j=-\alpha\} \,.
    \end{equation*}
It follows from \cite[Lemma 3.16]{Gfr25a} that
    \begin{align*}
        &\Sp(\partial\psi)\big((\ee{\hat x}j,\ee{\hat z}j),(\ee{\hat x{}^*}j,\ee{\hat z{}^*}j)\big)\\
        &\quad =\left\{\rge \kl{ \begin{pmatrix}I&0\\0&P\end{pmatrix},\begin{pmatrix}A^TA+\sigma B^TB&-\sigma B^TP\\
        -\sigma PB&\sigma P+W\end{pmatrix} } \Big\vert (P,W) \,\, \text{fulfills \eqref{EqP}}\right\}.
    \end{align*}
Furthermore, $\Sp^*(\partial\psi)\big((\ee{\hat x}j,\ee{\hat z}j),(\ee{\hat x{}^*}j,\ee{\hat z{}^*}j)\big)=\Sp(\partial\psi)\big((\ee{\hat x}j,\ee{\hat z}j),(\ee{\hat x{}^*}j,\ee{\hat z{}^*}j)\big)$, and all subspaces $L\in \Sp(\partial\psi)\big((\ee{\hat x}j,\ee{\hat z}j),(\ee{\hat x{}^*}j,\ee{\hat z{}^*}j)\big)$ are self-adjoint, i.e., $L=L^*$ \cite[Corollary 3.28]{GfrOut22}.

For computing the Newton direction, we can select an arbitrary subspace from the respective SCD, and we choose the one corresponding to the choice (in \eqref{EqP}) of
    \begin{equation}\label{EqI(z)}
        I^+\cup I^-=I(\ee{\hat z}j):=\{i\mv \ee{\hat z_i}j\not=0\} \,.
    \end{equation}
By \eqref{EqNewtonStep}, the Newton direction $(\ee{\triangle x}j,\ee{\triangle z}j)$ is given as solution of the linear system
    \begin{equation}\label{Eq_dx_dz}
        \begin{array}{ccccc}
            (A^TA+\sigma B^TB)\triangle x &-&\sigma B^TP\triangle z&=&-\ee{\hat x{}^*}j \,,
            \\
            -\sigma PB\triangle x&+&(\sigma P+W)\triangle z&=&-P\ee{\hat z{}^*}j \,.
        \end{array}
    \end{equation}
Taking into account \eqref{EqP} and \eqref{EqI(z)}, together with $\ee{\hat z{}^*}j=0$ we obtain from the second equation in \eqref{Eq_dx_dz} that
    \begin{equation}\label{Eq_dz}
        \triangle z_i=
        \begin{cases}
        0 \,, & \text{if} \,\,  i\not\in I(\ee{\hat z}j)\,,
        \\
        (B\triangle x)_i \,, &\text{if} \,\, i\in I(\ee{\hat z}j) \,.
        \end{cases}
    \end{equation}
Substituting $\triangle z$ into the first equation of \eqref{Eq_dx_dz} yields
    \begin{equation*}
        (A^TA+\sigma B^T(I-P)B)\triangle x=(A^TA+\sigma B^TWB)\triangle x=-\ee{\hat x{}^*}j \,,
    \end{equation*}
i.e., $\ee{\triangle x}j$ is a solution of the quadratic problem
    \begin{equation}\label{Eq_dx}
        \min_{\triangle x} \, \frac 12 \skalp{\triangle x, (A^TA+\sigma B^TWB)\triangle x}+\skalp{\ee{\hat x{}^*}j,\triangle x} \,.
    \end{equation}
However, for iterates far away from a solution, this might yield a poor Newton direction for the following reason: The components $\Psi_i$, $i\in I(\ee{\hat z}j)$ are continuously differentiable, in fact linear, around $\ee{\hat x}j$, and
    \begin{equation*}
        \nabla\Psi_i(\ee{\hat x}j)\ee{\triangle x}j=(B\ee{\triangle x}j)_i=\ee{\triangle z_i}j \,, \qquad \forall \, i\in I(\ee{\hat x}j) \,.
    \end{equation*}
Hence, using \eqref{EqPsi} and \eqref{Eq_dz}, we may conclude that
    \begin{equation*}
        \Psi_i(\ee{\hat x}j+\ee{\triangle x}j)=\Psi_i(\ee{\hat x}j)+\nabla\Psi_i(\ee{\hat x}j)\ee{\triangle x}j=\ee{\hat z_i}j+\ee{\triangle z}j_i \,,
        \forall \, i\in I(\ee{\hat z}j) \,,
    \end{equation*}
verifying that $\ee{\hat z_i}j$ and $\ee{\hat z_i}j+\ee{\triangle z}j_i$ have the same sign, i.e., $\ee{\triangle z}j_i/\ee{\hat z_i}j\geq -1$. However, whenever $\ee{\triangle z}j_i/\ee{\hat z_i}j< -1$, we can have a large approximation error
    \begin{equation*}
        \Psi_i(\ee{\hat x}j+\ee{\triangle x}j)-\Psi_i(\ee{\hat x}j)-\nabla\Psi_i(\ee{\hat x}j)\ee{\triangle x}j \,,
    \end{equation*}
due to the nonsmoothness of $\Psi$. Hence, we would ideally want to add the constraints
$(B\triangle x)_i/\ee{\hat z_i}j\geq-1$, $i\in I(\ee{\hat z}j)$ to the quadratic program \eqref{Eq_dx}, but this would considerably hinder its solution. Thus, we instead augment \eqref{Eq_dx} by the penalty term 
    \begin{equation*}
        \sum_{i\in I(\ee{\hat z}j)}\frac {\ee\rho j}2\kl{(B\triangle x)_i/{\ee{\hat z_i{}}j}}^2 \,,
    \end{equation*}
yielding the quadratic program
    \begin{equation}\label{Eq_dx1}
        \min \ee {\hat\psi}j(\triangle x):=\frac 12 \skalp{\triangle x, (A^TA+ B^T\ee Wj B)\triangle x}+\skalp{\ee{\hat x{}^*}j,\triangle x} \,,
    \end{equation}
where $\ee Wj$ is a diagonal matrix with entries
    \begin{equation}\label{Eq_Wj}
        \ee{W}j_{ii}
        =
        \begin{cases}
        \sigma \,, &\mbox{if $i\not\in I(\ee{\hat z}j)$} \,, 
        \\
        \frac {\ee\rho j}{{\ee{\hat z_i{}}j}^2} \,, & \mbox{otherwise} \,.
        \end{cases}
    \end{equation}
For updating the penalty parameter $\ee \rho j$, we employ the following strategy: If the solution $\ee {\triangle  x}j$ of \eqref{Eq_dx1} satisfies
    \begin{equation}\label{Eq_rj}
        \ee \chi j:=\min\Kl{\frac{(B\ee{\triangle x}j)_i}{\ee{\hat z_i}j}\mv i\in I(\ee{\hat z}j)}<-1 \,,
    \end{equation}
then we increase $\ee\rho j$, otherwise, if $\ee \chi j>-1$, then we decrease $\ee \rho j$.

Finally, having computed the direction $\ee{\triangle x}j$, we perform a line search along $\ee{\triangle x}j$ to obtain a decrease in the continuously differentiable function $\vartheta$ defined in \eqref{EqTheta}. We summarize the above considerations into the following algorithm:

\begin{algorithm}[Inexact regularized SCD \ssstar Newton method for \eqref{EqSubProbl}]\label{AlgSubProbl}\mbox{ }\\
Let parameters $\nu\in(0,1)$, $\chi_1<-1<\chi_2<0$, $0<\bar \chi_2<1<\bar \chi_1$, $\ee\rho0>0$, two  real sequences $\ee{\eps}j_A,\ee{\eps}j_N$ with elements  belonging to $(0,1)$, a requested tolerance $\eps>0$  and a starting point $\ee x0$ be given. Set $\ee z0:=\Psi(\ee x0)$.
\\ \vspace{-5pt} \\ \noindent 
For $j=0,1,\ldots$, perform the following steps until $\norm{\nabla_x\psi(\ee xj,\ee zj)}\leq \eps$

\begin{itemize}
    \item {\bf Step 1:} Approximation step:  By  applying the method of conjugate gradients (CG) to the quadratic program
        \begin{equation*}
            \min_x\psi(x,\ee zj) \,,
        \end{equation*}
    compute a point $\ee{\hat x}j$ satisfying
        \begin{equation*}
            \norm{\nabla_x\psi(\ee{\hat x}j,\ee zj)}\leq \ee{\eps}j_A\norm{\nabla_x\psi(\ee xj,\ee zj)} \,,
        \end{equation*}
    and set $\ee{\hat z}j:=\Psi(\ee{\hat x}j)$.
    \item {\bf Step 2:} Newton step: Using the CG method applied to \eqref{Eq_dx1}, compute a Newton direction $\ee{\triangle x}j$ satisfying
        \begin{equation*}
            \norm{\nabla \ee{\hat\psi}j(\ee{\triangle x}j)}\leq \ee{\eps}j_N\norm{\nabla \ee{\hat\psi}j(0)} \,.
        \end{equation*}
    \item {\bf Step 3:} Let $\ee lj$ be the first nonnegative integer $l$ such that
        \begin{equation*}
            \vartheta(\ee{\hat x}j+2^{-l}\ee{\triangle x}j)\leq \vartheta(\ee{\hat x}j) +\nu 2^{-l}\skalp{\nabla\vartheta(\ee{\hat x}l),\ee{\triangle x}j} \,,
        \end{equation*}
    and set $\ee x{j+1}:= \ee{\hat x}j+2^{-\ee lj}\ee{\triangle x}j$ and $\ee z{j+1}:=\Psi(\ee x{j+1})$.
    \item {\bf Step 4:} If $I(\ee{\hat z}j)=\emptyset$, set $\ee\rho{j+1}:=\ee\rho j$. Otherwise compute $\ee \chi j$ by \eqref{Eq_rj}, and
        \begin{equation*}
            \ee\rho{j+1}:=\begin{cases}\ee\rho j\min\{\frac{\ee \chi j}{\chi_1},\bar \chi_1\} \,, &\mbox{if $\ee \chi j<\chi_1$} \,,\\
            \ee\rho j \,, &\mbox{if $\chi_1\leq\ee \chi j\leq \chi_2$} \,, \\
            \ee\rho j\max\{\frac{\ee \chi j}{\chi_2},\bar \chi_2\} \,, &\mbox{if $\ee \chi j>\chi_2$} \,.
\end{cases}
        \end{equation*}
\end{itemize}
\end{algorithm}

\noindent
For the CG method in Steps~1 and 2, we additionally require the following details:

\begin{enumerate}
    \item In Step~1, we require that we start the CG method for minimizing $\psi(\cdot,\ee zj)$ with $\ee xj$, so that we perform at least one step of the CG method and $\ee{\hat x}j$ satisfies $\psi(\ee{\hat x}j,\ee zj)<\psi(\ee xj,\ee zj)$. In case when we use a preconditioned CG method, we demand that for all $j$ the same preconditioner is used for the matrix 
        \begin{equation*}
            \nabla^2\psi(\cdot,\ee zj)=A^TA+\sigma B^TB \,.
        \end{equation*}
    \item In Step~2, we require that the CG method is started with $\triangle x=0$ so that the outcome $\ee{\triangle x}j$ satisfies $\skalp{\ee{\hat x{}^*}j,\ee{\triangle x}j}<0$. Since $\nabla\vartheta(\ee{\hat x}j)=\ee{\hat x{}^*}j$ and $\vartheta$ is continuously differentiable, the line search in Step~3 is well-defined, and therefore
        \begin{equation*}
            \psi(\ee x{j+1},\ee z{j+1})=\vartheta(\ee x{j+1})<\vartheta(\ee {\hat x}j)=\psi(\ee {\hat x}j,\ee {\hat z}j) \,.
        \end{equation*}
\end{enumerate}

Note that Steps~1 and 2 are well defined, since the underlying quadratic programs are convex and possess a solution. Therefore, the CG method is capable to compute  solutions of the quadratic programs in finitely many steps, and thus the approximate optimality condition for terminating the CG method can be fulfilled.

\begin{theorem}\label{ThFiniteConvAlg2}
Algorithm~\ref{AlgSubProbl} stops after finitely many iterations.
\end{theorem}
\begin{proof}
Assume on the contrary that Algorithm~\ref{AlgSubProbl} does not terminate after finitely many iterations. In Step~1, we perform at least one step of the (preconditioned) CG method, and the outcome of the first iterate is of the form
    \begin{equation*}
        \ee{\tilde x}j=\ee xj -\tau Q\nabla_x\psi(\ee xj,\ee zj) \,,
    \end{equation*}
with    
    \begin{equation*}
        \tau=\frac{\skalp{\nabla_x\psi(\ee xj,\ee zj),Q\nabla_x\psi(\ee xj,\ee zj)}}{\skalp{Q\nabla_x\psi(\ee xj,\ee zj),(A^TA+\sigma B^TB)Q\nabla_x\psi(\ee xj,\ee zj)}} \,,
    \end{equation*}
where $Q$ is a symmetric positive definite $n\times n$ matrix used for preconditioning. For the respective function value, we obtain
    \begin{align*}
        \psi(\ee{\tilde x}j,\ee zj)-\psi(\ee xj,\ee zj)&=-\frac 12\frac{\skalp{\nabla_x\psi(\ee xj,\ee zj),Q\nabla_x\psi(\ee xj,\ee zj)}^2}{\skalp{Q\nabla_x\psi(\ee xj,\ee zj),(A^TA+\sigma B^TB)Q\nabla_x\psi(\ee xj,\ee zj)}}
        \\
        &\leq-\frac12 \frac{s^2\norm{\nabla_x\psi(\ee xj,\ee zj)}^2}S\leq-\frac {s^2\eps^2}{2S} \,,
    \end{align*}
where $s$ denotes the smallest eigenvalue of $Q$ and $S$ denotes the largest eigenvalue of $Q^T(A^TA+\sigma B^TB)Q$. Together with the inequalities
    \begin{equation*}
        \psi(\ee xj,\ee zj)-\frac {s^2\eps^2}{2S}\geq \psi(\ee{\tilde x}j,\ee zj)\geq \psi(\ee{\hat x}j,\ee zj)\geq \psi(\ee{\hat x}j,\ee {\hat z}j)>\psi(\ee x{j+1},\ee z{j+1}) \,,
    \end{equation*}
we conclude that $\lim_{j\to\infty}\psi(\ee xj,\ee zj)=-\infty$, contradicting the bound
    \begin{equation*}
        \inf_{x,z}\psi(x,z)=\inf_{x,z}\frac 12 \norm{Ax-b^\delta}^2-\frac {\norm{\zeta^*}^2}{2\sigma}+\frac \sigma2\norm{z-(Bx+\frac{\zeta^*}\sigma)}^2+ \alpha \norm{z}_1\geq-\frac {\norm{\zeta^*}^2}{2\sigma} \,.
    \end{equation*}
Hence, Algorithm~\ref{AlgSubProbl} terminates after finitely many iterations.
\end{proof}

If we run Algorithm~\ref{AlgSubProbl} with tolerance $\eps=0$ in order to produce an infinite number of iterations, we can easily infer from the proof of Theorem~\ref{ThFiniteConvAlg2} that
    \begin{equation*}
        \lim_{j\to\infty}\norm{\nabla_x\psi(\ee xj,\ee zj)}=\lim_{j\to\infty}\norm{\nabla_x\psi(\ee xj,\Psi(\ee xj))}=\lim_{j\to\infty}\norm{\nabla\vartheta(\ee xj)}=0 \,,
    \end{equation*}
and therefore, every accumulation point $\xb$ of the sequence $\ee xj$ together with $\zb:=\Psi(\bar x)$ minimizes $\psi$. Let us now state the following result on superlinear convergence.

\begin{theorem}\label{ThSuperLinNewton}
Assume that the function $\psi$ has a unique minimizer $(\xb,\zb)$ and assume that the subgradient $\zba:=\zeta^*+\sigma(B\xb-\zb)\in\partial \alpha\norm{\zb}_1$ fulfills the condition
    \begin{equation}\label{EqStrictCompl}
        \vert\zba_i\vert<\alpha,\ i\not\in I(\zb) \,.
    \end{equation}
Then the sequence $\ee xj$ produced by Algorithm~\ref{AlgSubProbl} with $\eps=0$ converges superlinearly to $\xb$, provided that $\lim_{j\to\infty}\ee \eps j_N=0$.
\end{theorem}
\begin{proof}
Clearly, $\zb=\Psi(\xb)$ and $\nabla\vartheta(\xb)=0$. Let the radius $r>0$ be chosen such that for every $\triangle x$ with $\norm{\triangle x}\leq r$ and for every $i\in I(\zb)$, the components $\zb_i$ and $(\zb+B\triangle x)_i$ have the same sign, and for every $i\not\in I(\zb)$ there holds $\vert (\zba+\sigma B\triangle x)_i\vert <\alpha$. Then it follows from \eqref{EqPsi_i} that for every such $\triangle x$, there holds
    \begin{equation*}
        \Psi_i(\xb+\triangle x)=\begin{cases}(\zb+B\triangle x)_i \,, &\mbox{if $i\in I(\zb)$} \,,
        \\
        \zb_i=0 \,, &\mbox{otherwise} \,.
        \end{cases}
    \end{equation*}
Hence, we obtain from \eqref{EqGradTheta} that
    \begin{align*}
        &\nabla\vartheta(\xb+\triangle x)= A^T(A(\xb+\triangle x)-b^\delta)+B^T\big(\zeta^*+\sigma(B(\xb+\triangle x)-\Psi(\xb+\triangle x))\big)
        \\
        &\quad =\nabla\vartheta(\xb)+(A^TA+\sigma B^TB)\triangle x+\sigma B^T\big(\Psi(\xb)-\Psi(\xb+\triangle x)\big)
        \\
        &\quad =\nabla\vartheta(\xb)+(A^TA+\sigma B^TB)\triangle x- \sigma \sum_{i\in I(\zb)}B_i^TB_i\triangle x 
        \\
        &\quad = \nabla\vartheta(\xb)+ (A^TA+\sigma B^T\bar W B)\triangle x
        =(A^TA+\sigma B^T\bar W B)\triangle x \,,
    \end{align*}
where $B_i$ denotes the $i$-th row of $B$ and $\bar W$ is the $n\times n$ diagonal matrix with entries
    \begin{equation*}
        \bar W_{ii}=\begin{cases}0 \,, &\mbox{if $i\in I(\zb)$} \,, \\
        1 \,, &\mbox{otherwise} \,. \end{cases}
    \end{equation*}
It follows that the function $\vartheta$ is quadratic on the ball $\B_r(\xb)$ around $\xb$ with radius $r$. Furthermore, the matrix $\nabla^2\vartheta(\xb)=A^TA+\sigma B^T\bar W B$ is nonsingular. To see this, assume on the contrary that there is a nonzero direction $\triangle x$ satisfying $(A^TA+\sigma B^T\bar W B)\triangle x=0$. We may assume that $\norm{\triangle x}<r$, implying
    \begin{equation*}
        \nabla\vartheta(\xb+\triangle x)=\nabla\vartheta(\xb)+(A^TA+\sigma B^T\bar W B)\triangle x=0 \,.
    \end{equation*}
By convexity of $\vartheta$, the point $\xb+\triangle x$ is another minimizer of $\vartheta$ and therefore, $(\xb+\triangle x,\Psi(\xb+\triangle x))$ is another minimizer of $\psi$, which is a contradiction. Hence, the matrix $A^TA+\sigma B^T\bar W B$ is nonsingular and, since it is obviously positive semidefinite, it is also positive definite. In what follows, we denote by $\bar s>0$ its smallest eigenvalue.

Thus, $\vartheta$ is strongly convex near $\xb$, and since $\vartheta(\ee xj)=\psi(\ee xj,\Psi(\ee xj))\to\vartheta(\xb)$ as $j\to\infty$, the sequence  $\ee xj$ converges to $\xb$. From $\vartheta(\ee{\hat x}j)<\vartheta(\ee xj)$, it follows that $\ee{\hat x}j$ also converges to $\xb$. Thus, we may conclude that for all $j$ sufficiently large, the points $\ee xj$ and $\ee{\hat x}j$ belong to $\B_r(\xb)$, and our definition of $r$ ensures that $I(\ee{\hat z}j)= I(\zb)$ holds for those $j$, and correspondingly we obtain that
    \begin{equation*}
        \ee {\hat x{}^*}j=\nabla\vartheta(\ee{\hat x}j)= (A^TA+\sigma B^T\bar W B)(\ee{\hat x}j-\xb) \,, \qquad\text{and}\qquad\norm{\ee {\hat x{}^*}j}\leq \bar S\norm{\ee{\hat x}j-\xb} \,,
    \end{equation*}
where $\bar S$ denotes the largest eigenvalue of $A^TA+B^T\bar WB$. Furthermore,
    \begin{equation*}
        0\leq\ee Wj_{ii}-\sigma\bar W_{ii}=\begin{cases}0 \,, &\mbox{if $i\not\in I(\zb)$} \,, 
        \\
        \ee\rho j/(\ee{\hat z}j_i)^2 \,, &\mbox{otherwise} \,, \end{cases}
    \end{equation*}
and we conclude that $\ee sj$, the smallest eigenvalue of $A^TA+B^T\ee WjB$, satisfies the inequality $\ee sj\geq\bar s>0$. Hence, it follows from the inequality
    \begin{align*}
        &\norm{(A^TA+B^T\ee WjB)\ee{\triangle x}j}-\norm{\ee {\hat x{}^*}j} \leq \norm{(A^TA+B^T\ee WjB)\ee{\triangle x}j+\ee{\hat x{}^*}j}
        \\
        &\qquad =
        \norm{\nabla \ee{\hat\psi}j(\ee{\triangle x}j)}
        \leq \ee{\eps}j_N\norm{\nabla \ee{\hat\psi}j(0)}=\ee{\eps}j_N\norm{\ee {\hat x{}^*}j}
    \end{align*}
that 
    \begin{equation*}
        \norm{\ee{\triangle x}j}\leq \frac{(1+\ee{\eps}j_N)}{\bar s}\norm{\ee {\hat x{}^*}j}\to 0\,,
        \qquad \text{as} \qquad j\to\infty \,. 
    \end{equation*}
Taking into account that $\ee{\hat z}j_i\to \zb_i\not=0$ for $i\in I(\ee{\hat z}j)=I(\zb)$, we may infer that $\lim_{j\to\infty}\ee \chi j=0$, $\lim_{j\to\infty}\ee \rho j=0$, and $\lim_{j\to\infty}\ee Wj=\sigma\bar W$. Since
    \begin{align*}
        &\ee\eps j_N\norm{\ee{\hat x{}^*}j}=\ee\eps j_N\norm{(A^TA+B^T\bar WB)(\ee{\hat x}j-\xb)}\geq\norm{(A^TA+B^T\ee WjB)\ee{\triangle x}j+\ee{\hat x{}^*}j}\\
        &\qquad =\norm{(A^TA+B^T\ee WjB)\ee{\triangle x}j+(A^TA+\sigma B^T\bar WB)(\ee{\hat x}j-\xb)}\\
        &\qquad \geq\norm{(A^TA+B^T\ee WjB)(\ee{\triangle x}j+\ee{\hat x}j-\xb)}-\norm{B^T(\ee Wj-\sigma \bar W)B(\ee{\hat x}j-\xb)} \,,
    \end{align*}
we obtain that
    \begin{equation}\label{EqAuxBnd1}
        \norm{\ee{\hat x}j+\ee{\triangle x}j-\xb}\leq\frac{\bar S\ee\eps j_N+\norm{B^T(\ee Wj-\sigma\bar W)B}}{\bar s}\norm{\ee{\hat x}j-\xb}\,.
    \end{equation}
Thus, we also have that $\ee{\hat x}j+\ee{\triangle x}j\in\B_r(\xb)$ for all $j$ sufficiently large, and we claim that $\ee lj=0$ in Step~3 of Algorithm~\ref{AlgSubProbl}, resulting in $\ee x{j+1}=\ee{\hat x}j+\ee{\triangle x}j$. Indeed, the use of the CG method ensures that
    \begin{equation*}
        \ee{\hat\psi} j(\ee{\triangle x}j)-\ee{\hat\psi} j(0)=\frac 12 \skalp{\nabla \ee{\hat\psi} j(0),\ee{\triangle x}j}= \frac 12 \skalp{\ee{\hat x{}^*}j,\ee{\triangle x}j}\leq 0 \,,
    \end{equation*}
and from 
    \begin{align*}
        \ee{\hat\psi} j(\ee{\triangle x}j)-\ee{\hat\psi} j(0)&=\vartheta(\ee {\hat x}j+\ee{\triangle x}j)-\vartheta(\ee {\hat x}j)+\frac 12\skalp{\ee{\triangle x}j,(\ee Wj-\sigma\bar W)\ee{\triangle x}j}
        \\
        &\geq \vartheta(\ee {\hat x}j+\ee{\triangle x}j)-\vartheta(\ee {\hat x}j) \,,
    \end{align*}
we infer that
    \begin{equation*}
        \vartheta(\ee {\hat x}j+\ee{\triangle x}j)-\vartheta(\ee {\hat x}j)\leq \frac 12 \skalp{\ee{\hat x{}^*}j,\ee{\triangle x}j}=\frac 12 \skalp{\nabla\vartheta(\ee{\hat x}j),\ee{\triangle x}j}<\nu\skalp{\nabla\vartheta(\ee{\hat x}j),\ee{\triangle x}j} \,,
    \end{equation*}
and consequently $\ee lj=0$. Finally, since
    \begin{align*}
        & \frac{\bar s}2\norm{\ee{\hat x}j-\xb}^2\leq\frac 12 \skalp{\ee{\hat x}j-\xb, (A^TA+\sigma B^T\bar W B)\ee{\hat x}j-\xb}=\vartheta(\ee{\hat x}j)-\vartheta(\xb)
        \\
        & \qquad \leq 
        \vartheta(\ee xj)-\vartheta(\xb)
        =\frac 12 \skalp{\ee xj-\xb, (A^TA+\sigma B^T\bar W B)\ee xj-\xb}\leq \frac{\bar S}2\norm{\ee  xj-\xb}^2 \,,
    \end{align*}
we obtain from \eqref{EqAuxBnd1} the inequality
    \begin{equation*}
        \norm{\ee x{j+1}-\xb}\leq \frac{\sqrt{\bar S}(\bar S\ee\eps j_N+\norm{B^T(\ee Wj-\sigma\bar W)B})}{\bar s^{3/2}}\norm{\ee  xj-\xb} \,,
    \end{equation*}
which proves the superlinear convergence of $\ee xj$ to $\xb$.
\end{proof}

\section{Numerical experiments}\label{sect_numerics}

In this section, we present results obtained with our \ssstar Newton approach for TV regularization on numerical experiments based on two (large-scale) tomographic imaging problems: X-ray computerized tomography (CT) and photoacoustic tomography (PAT). Furthermore, we compare our results to those obtained with state-of-the-art methods: ADMM, the Chambolle-Pock method, and the ``approximate'' iteration \eqref{method_BB}.

\subsection{Test setting I: X-ray computerized tomography (CT)}\label{subsect_setting_CT}

For the first test setting, we consider X-ray CT based on the 2D Radon transform \cite{Natterer_2001,Louis_1989}
    \begin{equation}\label{Radon_A}
    \begin{split}
        (\mathcal{A} x)(\sigma,\theta) :=  
        \int_\R x (\sigma \omega(\theta) + \tau \omega(\theta)^\perp) \, d\tau \,,
    \end{split}	
    \end{equation}
where $\omega(\theta) = (\cos(\theta),\sin(\theta))^T$ for $\theta \in [0,2\pi)$ and $\sigma \in \R$. The Radon transform can bee discretized into a finite-dimensional operator $A$ in many different ways, with both matrix-free and matrix-based implementations being available \cite{Hansen_2018,Astra_Toolbox}. Note that when using a matrix-free implementation, the corresponding routines for evaluating $Ax$ and $A^Tb$ have to be consistent (which excludes the use of Matlabs \texttt{radon} command here). In our numerical experiments, we use the walnut dataset released in \cite{FIP_Walnut_2015}, which contains sinograms $b$ of different resolution, as well as the corresponding measurement matrix $A$ discretizing $\mathcal{A}$.

\begin{figure}[ht!]
    \centering
    \includegraphics[width=\textwidth]{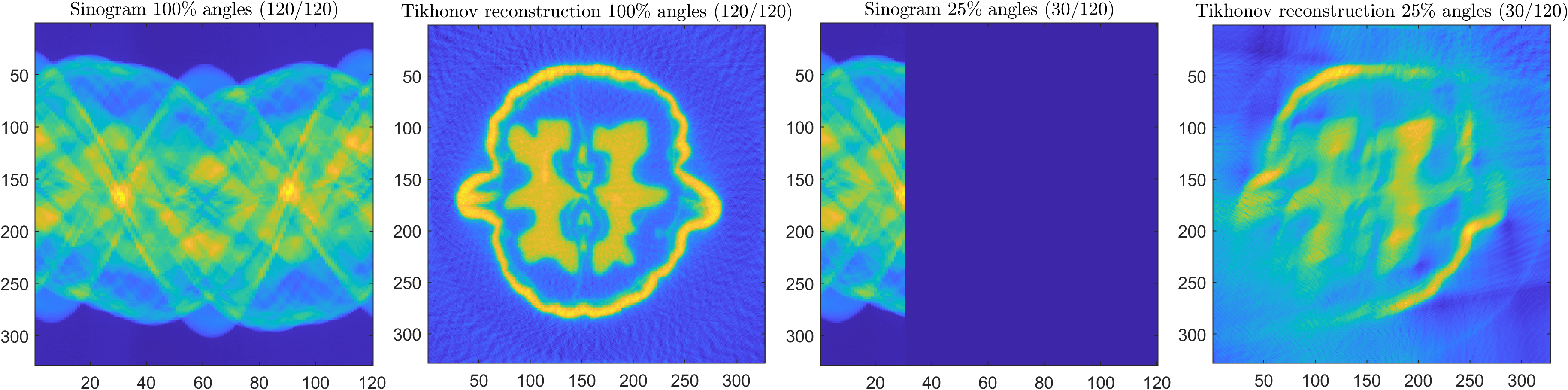}
    \caption{Test setting I (CT): Examples of sinogram data ($100\%$ and $25\%$ angles) and standard Tikhonov reconstructions with $L^2$-penalty (no TV). Adapted/reproduced from the dataset \url{https://fips.fi/open-datasets/x-ray-tomographic-datasets/tomographic-x-ray-data-of-a-walnut}, \mbox{CC~BY~4.0.} See also \cite{FIP_Walnut_2015}.}
    \label{fig_CT_Setting}
\end{figure}

In our experiments, we considered several different (coarse and fine) spatial resolutions ($82\times 82$, $164\times 164$, $328\times 328$ pixels), numbers of parallel rays ($20$ and $120$), and limited-angle settings ($25\%$, $50\%$, $75\%$, $100\%$). Figure~\ref{fig_CT_Setting} depicts two sinograms, corresponding to $100\%$ and $25\%$ angles, for resolution $328 \times 328$ and $120$ parallel rays, as well as a simple Tikhonov inversion with an $L^2$ instead of a TV penalty term ($\alpha = 10$).

\subsection{Test setting II: Photoacoustic tomography (PAT)}\label{subsect_setting_PAT}

For the second test setting, we consider a tomographic inverse problem relating to PAT. PAT is a hybrid imaging modality which leverages the photoacoustic effect, combining optical contrast with ultrasonic resolution \cite{li2009,beard2011,Kuchment_Kunyansky_2008,wang2011photoacoustic}. 

In a typical PAT experiment, a short optical pulse illuminates the target, leading to a thermal expansion in regions with optical absorption. This process generates an initial pressure distribution inside the target medium. The magnitude of this pressure depends on the local optical absorption and thermal expansion properties of the tissue. Following this excitation, the induced pressure distribution relaxes as an acoustic wave. Due to the differences in time scales between the optical and acoustical phenomena, the generation of the initial pressure is commonly assumed to be instantaneous. The imaging task in PAT is therefore to recover the initial pressure distribution from measurements of the propagating acoustic waves.

Mathematically, this problem is modelled as follows: In an acoustically homogeneous and non-attenuating medium, the acoustic pressure $p(r,t)$ satisfies the wave equation
    \begin{equation}\label{eq:wave-eq}
        \kl{ \frac{1}{c^2(s)} \frac{\partial^2}{\partial t^2} - \Delta_s } p(s,t) = 0 \,,
        \qquad 
        \forall \, (s,t) \in \mathbb{R}^d \times (0,T] \,,
    \end{equation}
subject to the initial conditions $p(s,0) = p_0(s)$ and $\partial_t p(s,0) = 0$, where $p_0(s)$ denotes the initial pressure distribution and $c(s)$ is the sound speed of the medium, which is assumed to be known. Note that in many applications, as well as in our numerical experiments conducted below, $c$ is also assumed to be constant.

\begin{figure}[ht!]
    \centering
    \includegraphics[width=0.9\textwidth]{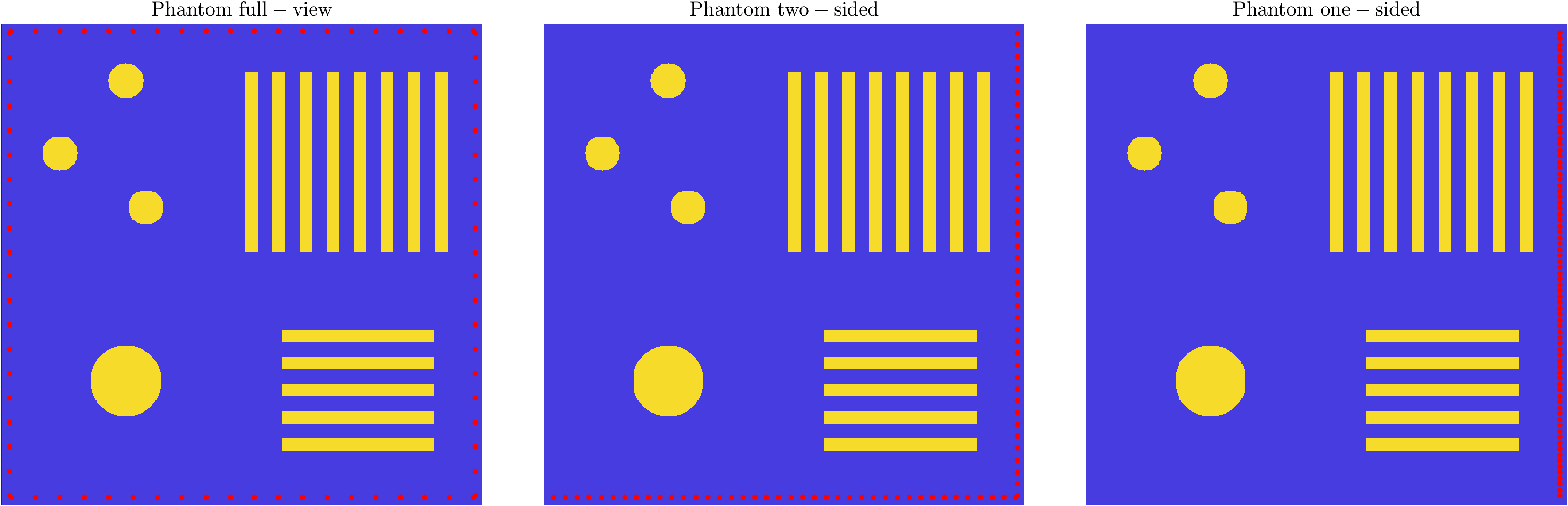}
    \caption{Test setting II (PAT): Ground truth $p_0$ and sensor locations (red dots).}
    \label{fig_PAT_Setting}
\end{figure}

PAT data consist of time-dependent pressure measurements recorded at sensor locations on the boundary $\partial \Omega$ of the target. Figure~\ref{fig_PAT_Setting} illustrates the three different senor configurations used in our experiments, correspondingly referred to as ``full-view'', ``one-sided'', and ``two-sided'', mimicking a photoacoustic-setup with a Fabry-Perot based sensor head (see, e.g., \cite{Ellwood_2017,Zhang_Laufer_Beard_2008}), as well as the ground truth pressure $p_0(s)$. For a finite set of point-like sensors at positions $\{s_i\}_{i=1}^M \subset \partial \Omega$, which in Figure~\ref{fig_PAT_Setting} are illustrated as red dots ($80$ total, equally spaced), these measurements are mathematically given by
    \begin{equation*}
        p(s_i,t)\,,  \qquad \forall \, t \in [0,T]\,, i = 1 \,, \dots \,, M \,.
    \end{equation*}
This measurement process can mathematically be encoded by the forward operator
    \begin{equation*}
        K : p_0 \mapsto \Kl{ p(s_i,t) }_{i=1}^{N_s},
    \end{equation*}
where $p(s,t)$ is the solution of the wave equation \eqref{eq:wave-eq} with initial data $p_0$. Hence,
    \begin{equation*}
        (K p_0)(i,t) := p(s_i,t), 
        \qquad \forall \, (i,t) \in \{1,\dots,N_s\} \times [0,T] \,,
    \end{equation*}
which can be seen to be a linear operator in $p_0$. 
The solution to the wave-equation \eqref{eq:wave-eq} can be numerically approximated using a pseudospectral method. For this, we use the MATLAB \texttt{k-Wave} toolbox \cite{treeby2010a}, which simulates acoustic wave propagation by solving a first-order formulation of the wave equation in Fourier domain (using a perfectly-matched layer boundary condition) and finite differences in time. Given an initial pressure distribution $p_0$, \texttt{k-Wave} computes the corresponding pressure field $p(s,t)$ and returns the simulated time-series sensor data at prescribed sensor locations, thus providing a discrete approximation of the forward operator $K$.

In our numerical experiments, synthetic sensor data is created using a fine spatial discretization ($1000 \times 1000$ pixels) in \texttt{k-Wave}. The temporal discretization is set according to the Nyquist frequency ($100$ MHz), and is given in matrix-form $ (p_t)_{ij} = (K p_0) (r_i,t_j)$. To avoid an inverse crime, a coarser reconstruction grid is used, and the data is interpolated to fit the subsequent Nyquist frequency. Then, $5\%$ relative random noise is added to the data. The adjoint operator $K^*$ is defined via the $L^2$ inner product relation, and in the present setting, can be interpreted as a time-reversal process: Given measured data $p(i,t)$ at the sensor locations $\{r_i\}_{i=1}^M$, one defines a wave field $q(r,t)$ as the solution of the wave equation
    \begin{equation*}
        \left( \frac{1}{c^2(s)} \frac{\partial^2}{\partial t^2} - \Delta_s \right) q(s,t) = 0 \,,
    \end{equation*}
supplemented with final conditions $q(s,T) = 0$ and $\partial_t \, q(s,T) = 0$. In the case of point-like sensors, this can be modeled by imposing time-reversed source terms at the detector locations. With this, the action of the adjoint operator is then given by evaluating the resulting field at the initial time, i.e.,
    \begin{equation*}
        K^* g = q(s,0) \,,
    \end{equation*}
which can be interpreted as backprojecting the measured signals into the domain.  In practice, the adjoint $K^*$ is implemented numerically using the same wave propagation solvers as for the forward problem. In the \texttt{k-Wave} toolbox, this can be achieved by a time-reversal procedure with time-reversed sensor data as the source, and simulating the wave propagation backward in time. This yields a discrete approximation of $K^*$ which is (numerically) consistent with the approximate forward operator $K$ \cite{Arridge_2016}.

Note that in our discrete setting of \eqref{EqLSProbl}, this means that for the matrix-vector products $Ax$ and $A^T b$, we use the matrix-free evaluations of $K$ and $K^*$ described above. Computing a matrix representation for $K$ is only feasible for coarse discretizations.

\subsection{Implementation and computational aspects}\label{subsect_implementation}

In all of our numerical experiments, the vector $x\in\R^n$ corresponds to a 2D-image with $n_{\row}\times n_{\col}$ pixels, i.e., $n=n_{\row}n_{\col}$. Since we are interested in TV regularization, $Bx$ should approximate the gradient of this image. For this, we identify $\R^n$ with $\R^{n_{\row}\times n_{\col}}$, and define the matrix $B$ of size $\big(n_{\row}(n_{\col}-1)+ (n_{\row}-1)n_{\col}\big)\times n$ such that the vector $Bx$ has the elements
    \begin{equation*}
        x_{i,j+1}-x_{i,j} \,,
        \qquad 
        \forall \, i=1\,,\dots\,, n_{\row}\,, \quad \forall \, j=1\,,\ldots\,,n_{\col}-1 \,,
    \end{equation*}
and
    \begin{equation*}
        x_{i+1,j}-x_{i,j} \,, 
        \qquad \forall \, i=1 \,,\dots \,, n_{\row}-1 \,, \quad \forall \, j=1\,,\dots \,, n_{\col} \,.
    \end{equation*}
    
Concerning the initialization of Algorithm~\ref{AlgALM}, note that we choose
    \begin{equation*}
        \ee \sigma0=10\frac{\lambda_{max}(A^TA)}{\lambda_{max}(B^TB)} \,,
    \end{equation*}
where $\lambda_{max}(C)$ denotes the largest eigenvalue of the matrix $C$. Then, given the starting values $\ee x0$ and $\ee{\zeta^*}0$, according to \eqref{EqArgMinLag} we compute
    \begin{equation*}
        \ee z0=\argmin_z\lag_{\ee\sigma0}(\ee x0,z,\ee{\zeta^*}0)=\Prox{\alpha/\ee\sigma0}\Big(B\ee x0+\frac {\ee{\zeta^*}0}{\ee\sigma0}\Big)
    \end{equation*}
and use this result to recompute
    \begin{equation*}
        \ee{\zeta^*}0\leftarrow\ee{\zeta^*}0+\ee\sigma0(B\ee x0-\ee z0)\in \partial\alpha\norm{\cdot}_1(\ee z0) \,.
    \end{equation*}
    
Concerning the stopping of Algorithm~\ref{AlgALM}, note that at the start of each iteration $k$, the fulfillment of the first-order optimality conditions \eqref{EqFO_ExtProbl} is measured by
    \begin{equation}\label{Eq_rk}
        \ee rk :=\Big(\norm{A^T(A\ee xk-b^\delta)+B^T\ee{\zeta^*}k}^2+\gamma_{\op{scale}}^2\norm{B\ee xk-\ee zk}^2\Big)^{1/2} \,,
    \end{equation}
where the scaling factor
    \begin{equation}\label{EqScalingFactor}
        \gamma_{\op{scale}}:=\lambda_{\max}(A^TA)/\sqrt{\lambda_{\max}(B^TB)} 
    \end{equation}
ensures that the quotient $\ee rk/\ee r0$ is independent of transformations of the form 
    \begin{equation*}
        (x,A,B)\to(x/\eta,\eta A,\eta B)\,,
        \qquad \text{and} \qquad
        (B,\alpha)\to (\eta B,\alpha/\eta) \,.
    \end{equation*}
with positive scalar $\eta$. For a given tolerance $\eps_{\op{Opt}}$, we then stop Algorithm~\ref{AlgALM} if
    \begin{equation}\label{EqTermininateALM}
        \ee rk\leq\eps_{\op{Opt}}\ee r0 \,.
    \end{equation}
Finally, for the sequence $\gamma_l$ in Algorithm~\ref{AlgALM} we chose $\gamma_l=5/(5+l)$, and we set $\beta=0.5$.

For the computation of the iterates $(\ee x{k+1},\ee z{k+1})$ according to \eqref{EqApprMinLaq} in Algorithm~\ref{AlgALM}, we use Algorithm~\ref{AlgSubProbl} with the stopping tolerance
    \begin{equation*}
        \eps=\ee \eps k :=\min\Kl{2^{-(k+1)}\norm{\nabla_x\lag_{\ee\sigma0}(\ee x0,\ee z0,\ee{\zeta^*}0)},\ 0.1\norm{\nabla_x\lag_{\ee\sigma k}(\ee xk,\ee zk,\ee{\zeta^*}k)}} \,.
    \end{equation*}
The tolerances in Steps~1 and 2 of Algorithm~\ref{AlgSubProbl} are set to $\ee{\eps_A}j = \ee{\eps_N}j = 0.1$, and for the remaining parameters we use $\nu=0.1$, $\xi_1=-1.2$, $\bar\xi_1=4$, $\xi_2=-0.8$, and $\bar\xi_2=0.25$.
The parameter $\ee \rho 0$ was chosen as
    \begin{equation*}
        \ee \rho 0 := 10^{-4}\norm{\ee{\hat z{}}0}^2/ \abs{ I(\ee{\hat z{}}0) } \,.    
    \end{equation*}

As noted above, we compare our method with the Chambolle-Pock method \eqref{Chambolle_Pock}, which in our setting \eqref{EqLSProbl} takes the form as given in Algorithm~\ref{AlgChPo}.

\begin{algorithm}[Chambolle-Pock algorithm]\label{AlgChPo}\mbox{ }\\
Choose $\tau,\sigma>0$, $\theta\in[0,1]$, $(\ee x0\, \ee {z^*}0)\in\R^n\times\R^l$ and set $\ee {\bar x}0=\ee x0$.
\\ \vspace{-8pt} \\ \noindent
For $k=0,1,\ldots$, perform the following calculations
    \begin{align}
        \label{EqChPoSt1}&\ee {z^*}{k+1} = {\rm prox}_{\delta_{\alpha\B_{\infty}}}(\ee {z^*}k+\sigma B\ee {\bar x}k) \,,
        \\
        &\label{EqChPoSt1a}\ee zk = \frac{\ee{z^*}k-\ee{z^*}{k+1}}\sigma + B\ee{\bar x}k \,,
        \\
        &\label{EqChPoSt1b} \ee rk =\Big(\norm{A^T(A\ee xk-b^\delta)+B^T\ee{z^*}{k+1}}^2+\gamma_{scale}^2\norm{B\ee xk-\ee zk}^2\Big)^{1/2}
        \\
        \label{EqChPoSt2}&\ee  x{k+1} = \ee xk -\tau(I+\tau A^TA)^{-1}\big(A^T(A\ee xk-b^\delta)+B^T\ee {z^*}{k+1}\big)\,,
        \\
        \label{EqChPoSt3}&\ee {\bar x}{k+1} = \ee  x{k+1} +\theta(\ee  x{k+1}-\ee xk) \,.
    \end{align}
\end{algorithm}

In Algorithm~\ref{AlgChPo}, the scaling parameter $\gamma_{\op{scale}}$ is chosen as in \eqref{EqScalingFactor}, and ${\rm prox}_{\delta_{\alpha\B_{\infty}}}$ denotes the proximal mapping of the indicator function of the $\ell_\infty$-ball in $\R^l$ with radius $\alpha$, i.e., the $i$-th component of $\ee {z^*}{k+1}$ in \eqref{EqChPoSt1} is given by
    \begin{equation*}
        \ee{z^*_i}{k+1}=\begin{cases}
        \alpha \,, &\text{if} \,\, \ee {z^*}k+\sigma B\ee {\bar x}k)_i>\alpha \,,
        \\
        (\ee {z^*}k+\sigma B\ee {\bar x}k)_i \,, &\text{if} \,\, \vert(\ee {z^*}k+\sigma B\ee {\bar x}k)_i\vert\leq\alpha \,,
        \\
        -\alpha \,, &\text{if} \,\, (\ee {z^*}k+\sigma B\ee {\bar x}k)_i<-\alpha \,.
    \end{cases}
    \end{equation*}
By the definition of the proximal mapping, it follows that 
    \begin{equation*}
        0\in \ee{z^*}{k+1}-\ee {z^*}k-\sigma B\ee {\bar x}k+\partial \delta_{\alpha\B_{\infty}}(\ee{z^*}{k+1}) \,,
    \end{equation*} 
and, since $\delta_{\alpha\B_{\infty}}$ is the conjugate function to $\alpha\norm{\cdot}_1$, we obtain that
    \begin{equation*}
        \ee {z^*}{k+1}\in \partial \alpha\norm{\ee {z^*}k-\ee{z^*}{k+1}+\sigma B\ee {\bar x}k}_1 =\partial \alpha\norm{\ee zk}_1 \,,
    \end{equation*}
where we have used that $\partial\alpha\norm{\sigma \ee zk}_1=\partial\alpha\norm{\ee zk}_1$ due to the special form of the subdifferential of the $\ell_1$-norm. Hence, $\ee rk$ in \eqref{EqChPoSt1b} measures the violation of the first-order optimality conditions at $(\ee xk,\ee zk,\ee {z^*}{k+1})$ in the same way as in \eqref{Eq_rk}.

Concerning the stopping of Algorithm~\ref{AlgChPo}, we terminate the iteration as soon as
    \begin{equation}\label{EqTermininateChPo}
        \ee rk\leq\eps_{\op{Opt}}^{\op{ChPo}}\ee r0 \,,
    \end{equation}
for some given tolerance $\eps_{\op{Opt}}^{\op{ChPo}}$. The parameters are chosen as $\theta=1$, 
\begin{equation}\label{EqChPo_tau}\tau= 4/\lambda_{max}(A^TA),
\end{equation}and $\sigma=1/(\tau\lambda_{max}(B^TB))$. The iterate $\ee x{k+1}$ in \eqref{EqChPoSt2} is computed as $\ee x{k+1}=\ee xk-\tau\triangle x$, where $\triangle x$ is computed by the CG method as an approximate solution of 
    \begin{equation*}
        (I+\tau A^TA)\triangle x=A^T(A\ee xk-b^\delta)+B^T\ee {z^*}{k+1} \,,
    \end{equation*}
satisfying
    \begin{small}
    \begin{equation*}
        \norm{(I+\tau A^TA)\triangle x-\big(A^T(A\ee xk-b^\delta)+B^T\ee {z^*}{k+1}\big)}\leq 10^{-3}\norm{A^T(A\ee xk-b^\delta)+B^T\ee {z^*}{k+1}} \,. 
    \end{equation*}
    \end{small}

In addition to the Chambolle-Pock method, we also compare our approach to the ADMM method \cite{BoydADMM}. In our setting \eqref{EqLSProbl}, this method takes the form of Algorithm~\ref{AlgADMM}.

\begin{algorithm}[ADMM]\label{AlgADMM}\mbox{ }\\
Choose $\sigma>0$ and starting point $(\ee x0,\ee{z^*}0)\in\R^n\times\R^l$. Compute
    \begin{equation*}
    \begin{split}
        \ee z0 &= \argmin_z \lag_\sigma(\ee x0,z,\ \ee{z^*}0) = \Prox{\alpha/\sigma}\kl{B\ee x0 +\frac{\ee{z^*}0}\sigma} \,,
        \\ 
        \ee{z^*}0 &= \ee{z^*}0+\sigma(B\ee x0-\ee z0) \,.
    \end{split}
    \end{equation*}
For $k=0,1,\ldots$, perform the following calculations
    \begin{align}
        \nonumber  &\ee rk=\Big(\norm{A^T(A\ee xk-b^\delta)+B^T\ee{z^*}k}^2+\gamma_{scale}^2\norm{B\ee xk-\ee zk}^2\Big)^{1/2} \,,
        \\
        \label{EqADMM_x_k+1}  &\ee x{k+1} =\argmin_x \lag_\sigma(x,\ee zk,\ee{z^*}k) \,,
        \\
        \nonumber &\ee z{k+1} = \argmin_z \lag_\sigma(\ee x{k+1},z,\ee{z^*}k) = \Prox{\alpha/\sigma}\kl{B\ee x{k+1} +\frac{\ee{z^*}k}\sigma } \,,
        \\
        \nonumber &\ee {z^*}{k+1}= \ee{z^*}k+\sigma(B\ee x{k+1}-\ee z{k+1}) \,.
    \end{align}
\end{algorithm}

Concerning the stopping of Algorithm~\ref{AlgADMM}, we terminate the iteration as soon as
    \begin{equation}\label{EqTermininateADMM}
        \ee rk\leq\eps_{\op{Opt}}^{\op{ADMM}}\ee r0\,,
    \end{equation}
for some given tolerance $\eps_{\op{Opt}}^{\op{ADMM}}$. The parameter $\sigma$ is chosen as 
    \begin{equation}\label{EqADMM_sigma}
        \sigma=0.25\lambda_{max}(A^TA)/\lambda_{max}(B^TB) \,,
    \end{equation}
and the new iterate $\ee x{k+1}$ in \eqref{EqADMM_x_k+1} is calculated by $\ee x{k+1}=\ee xk+\triangle x$, where $\triangle x$ is computed by the CG method as an approximate solution of the linear system
    \begin{equation*}
        (A^TA+\sigma B^TB)\triangle x=-\nabla_x\lag_\sigma(\ee xk,\ee zk,\ee {z^*}k) \,,
    \end{equation*}
satisfying
    \begin{equation*}
        \norm{(A^TA+\sigma B^TB)\triangle x+\nabla_x\lag_\sigma(\ee xk,\ee zk,\ee {z^*}k)}\leq 10^{-3}\norm{\nabla_x\lag_\sigma(\ee xk,\ee zk,\ee {z^*}k)} \,.
    \end{equation*}
    
Note that the proper choice of the parameters $\theta,\tau,\sigma$ in Algorithm~\ref{AlgChPo}, and of $\sigma$ in Algorithm~\ref{AlgADMM} is crucial for the speed of convergence of these two algorithms. The coefficients 4 and 0.25 appearing in \eqref{EqChPo_tau} and \eqref{EqADMM_sigma}, respectively, were estimated by a grid search and yield good rates of convergence for all considered test settings.

Concerning the regularization parameter $\alpha$ in \eqref{EqLSProbl}, we used the following heuristic: 
A solution $(\xad)$ of \eqref{EqLSProbl} satisfies the first-order optimality condition
    \begin{equation*}
        0\in A^T(A \xad- b^\delta)+\alpha B^T z_\alpha^* \,, \qquad  z_\alpha^*\in\partial \norm{B \xad}_1 \,,
    \end{equation*}
implying
    \begin{equation}\label{EqAlpha}
        \alpha = \frac{\norm{A^T(A \xad -b^\delta)}}{\norm{B^Tz_\alpha^*}} \,.
    \end{equation}
Now note that according to the discrepancy principle, $\alpha$ should be chosen such that $\norm{A \xad- b^\delta}\approx\delta$, and that under the assumption that the noise $b-b^\delta$ is normally distributed, we may estimate the numerator in \eqref{EqAlpha} by
    \begin{equation*}
        \sum_{i=1}^t\delta\frac{\norm{A^T b_i}}{t\norm{b_i}} \,,
    \end{equation*}
where $b_i$, $i=1\,,\ldots\,,t$ denote random vectors with ${\cal N}(0,1)$ distributed elements. Now in order to estimate the denominator in \eqref{EqAlpha}, we utilize the facts that $z_\alpha^*$ is contained in the $l_\infty$ unit ball, and that $\norm{z_\alpha^*}_\infty=1$ provided $B\xad \neq 0$, and we replace $\norm{B^Tz_\alpha^*}$ by
    \begin{equation*}
        \sum_{i=1}^t\frac{\norm{B^Tz_i^*}}{t\norm{z_i^*}_\infty} \,, 
    \end{equation*}
where $z_i^*$, $i=1 \,, \ldots \,, d$ are random vectors with elements uniformly drawn from $[-0.5,0.5]$. This approach yields the following formula for the regularization parameter $\alpha$: 
    \begin{equation}\label{EqAlphaOpt}
        \alpha = \delta\kl{\sum_{i=1}^t\frac{\norm{A^T b_i}}{\norm{b_i}}} \Big/\kl{\sum_{i=1}^t\frac{\norm{B^Tz_i^*}}{\norm{z_i^*}_\infty}}\,.
    \end{equation}
In our numerical experiments presented below, we used the sample size $t=10$. This heuristic choice of the regularization parameter is easy to implement and yields good results in practice. However, it is far beyond the scope of this paper to give a full theoretical justification for it. We observed that the performance of the considered numerical methods is not very sensitive with respect to small changes in $\alpha$ and the obtained results remain nearly the same for repeated trials.

Concerning the computational environment, we have implemented all algorithms in Matlab, and run them on a single core (Intel Haswell CPU, Xeon E5-2630v3, 2.4Ghz) of the HPC Cluster RADON~1 (see \url{https://www.oeaw.ac.at/ricam/hpc}).

Finally, note that in most of the numerical experiments conducted below, it appears that the solution $(\xb,\zb)$ of \eqref{EqExtProblem} is unique, but that the problem is dually degenerated in the sense that there are multiple subgradients $z^*$ fulfilling the first-order optimality conditions \eqref{EqFO_ExtProbl}. Now let $(\xb,\zb)$ be a solution of \eqref{EqExtProblem} and consider the undirected graph $G=(V,E)$ with vertices $(i,j)$, $i=1,\ldots,n_{\row}$, $j=1,\ldots,n_{\col}$, and edges
    \begin{equation*}
        E=\Kl{[(i,j),(i+1,j)]\mv \xb_{i+1,j}-\xb_{i,j}=0}
        \cup
        \Kl{[(i,j),(i,j+1)]\mv \xb_{i,j+1}-\xb_{i,j}=0} \,,
    \end{equation*}
where we have identified $\xb$ with an element of $\R^{n_{\row}\times n_{\col}}$. Hence, the edges of the graph correspond to the rows of the matrix $B$ where $\zb=B\xb$ is zero. Now as soon as a cycle 
    \begin{equation*}
        \big((i_1,j_1),(i_2,j_2),\ldots,(i_{k+1},j_{k+1})=(i_1,j_1)\big)
    \end{equation*}
in $G$ exists, the corresponding rows of $B$ are linearly dependent: Taking the weight of the row of $B$ related to the edge $[(i_t,j_t),(i_{t+1},j_{t+1}]$, $t=1,\ldots,k$ to be $+1$ if $i_{t+1}\geq i_t$ and $j_{t+1}\geq j_t$, and $-1$ otherwise, it is easy to see that the weighted sum of the rows of $B$ is zero. Hence, it follows that in this case, whenever a subgradient $\bar \zb^*\in\ri \partial\alpha\norm{\zb}$ fulfilling \eqref{EqFO_ExtProbl} exists, then $S_{\rm FO}$ is not a singleton. Note that it is quite likely that a cycle in $G$ exists, e.g., when four pixels $\xb_{i,j},\xb_{i+1,j},\xb_{i+1,j+1},\xb_{i,j+1}$ forming a square have the same value. Further, for almost all $(z,z^*)\in\gph \partial\norm{\cdot}_1\subset \R^l\times\R^l$ (wrt.\ the $l$-dimensional Hausdorff measure), we have $z^*\in\ri\partial\norm{z}_1$ and therefore, we expect that problem \eqref{EqExtProblem} is dually degenerated.

\begin{remark}
In case problem \eqref{EqExtProblem} is dually degenerated, there exist $(\zeta^*,\sigma)$ such that for the problem $\min \lag_\sigma(\cdot,\cdot,\zeta^*)$
the strict-complementarity-type condition \eqref{EqStrictCompl} of Theorem~\ref{ThSuperLinNewton} is violated and therefore, we cannot guarantee superlinear convergence of Algorithm \ref{AlgSubProbl}. On the other hand, consider $(x,z)\in\argmin \lag_\sigma(\cdot,\cdot,\zeta^*)$ and assume that the set 
    \begin{equation*}
        I^0:=\Kl{i\not\in I(z)\mv \vert \abs{\zeta^*+\sigma(Bx-z)}_i\vert=\alpha }
    \end{equation*}
is not empty. It is easy to verify that for perturbations $\triangle\zeta^*$ satisfying
    \begin{equation*}
        \triangle \zeta^*_i\begin{cases}=0 \,, &\mbox{if $i\not\in I^0$ \,, }\\
        >0 \,, &\mbox{if $i\in I^0$ and $\big(\zeta^*+\sigma(Bx-z)\big)_i=\alpha$ \,, }\\
        <0 \,, &\mbox{if $i\in I^0$ and $\big(\zeta^*+\sigma(Bx-z)\big)_i=-\alpha$ \,,}\end{cases}
    \end{equation*}
we have 
    \begin{equation*}
        (x,z+\triangle \zeta^*/\sigma)\in\argmin \lag_\sigma(\cdot,\cdot,\zeta^*+\triangle\zeta^*) \,,
    \end{equation*}
and 
    \begin{equation*}
        \Kl{i\not\in I(z+\triangle\zeta^*/\sigma)\mv \abs{ \big(\zeta^*+\triangle \zeta^*+\sigma(Bx-(z+\triangle \zeta^*/\sigma))\big)_i}=\alpha }=\emptyset \,,
    \end{equation*}
i.e., strict complementarity holds. We claim (without proof) that the assumptions of Theorem \ref{ThSuperLinNewton} generically hold in the following sense: If \eqref{EqExtProblem} has a unique solution $(\xb,\zb)$, then for each $\sigma_0>0$ there is an open set $U$ containing $\{z^*\mv (\xb,\zb,z^*)\mbox{ fulfills }\eqref{EqFO_ExtProbl}\}$ and an open set $V\subset U\times(\sigma_0,\infty)$ which is dense in $U\times(\sigma_0,\infty)$ such that for every $(\zeta^*,\sigma)\in U$ the set $\argmin\lag_\sigma(\cdot,\cdot,\zeta^*)$ is the singleton $\{\big(x(\zeta^*,\sigma),z(\zeta^*,\sigma)\big)\}$, and there holds
    \begin{equation*}
        \abs{\kl{\zeta^*+\sigma(Bx(\zeta^*,\sigma)-z(\zeta^*,\sigma))}_i }<\alpha\,,\qquad i\not\in I\big(z(\zeta^*,\sigma)\big),\ (\zeta^*,\sigma)\in V \,.
    \end{equation*}
\end{remark}

\subsection{Numerical results for setting I: X-ray CT}

We now present numerical results in the CT test setting described in Section~\ref{subsect_setting_CT}. First, note that the noise level $\delta$ in \eqref{EqAlphaOpt} was estimated by $\delta = 0.03$, and that we always used the initial guess $\ee x0=0$ and $\ee{z^*}0=0$. In the termination criteria \eqref{EqTermininateALM}, \eqref{EqTermininateChPo}, and \eqref{EqTermininateADMM}, we chose the tolerances $\eps_{\op{Opt}}=\eps_{\op{Opt}}^{\op{CP}}=\eps_{\op{Opt}}^{\op{ADMM}}=10^{-6}$, which corresponds to approximately 2 significant digits. Since the residual $\ee rk$ only has limited significance, we also considered the relative distance to the solution $\xb_{\op{ALM}}$ found by Algorithm~\ref{AlgALM} with stopping tolerance $\eps_{\op{Opt}}=10^{-9}$, measured by
    \begin{equation}\label{EqRelErr}
        \ee {\rm err}k_2:=\frac{\norm{\ee xk-\xb_{\op{ALM}}}_2}{\norm{\xb_{\op{ALM}}}_2} \,,
        \qquad
        \text{and}
        \qquad 
        \ee {\rm err}k_\infty := \frac{\norm{\ee xk-\xb_{\op{ALM}}}_\infty}{\norm{\xb_{\op{ALM}}}_\infty} \,.
    \end{equation}
These errors are more suitable convergence measures for the considered algorithms than the distance to the ground-truth solution, which is both unknown and not necessarily even the minimizer of \eqref{EqLSProbl}, while $\xb_{\op{ALM}}$ is a good approximation (up to $\eps_{\op{Opt}}=10^{-9}$).

\begin{table}[ht!]
\centering
\begin{tabular}{|c|c||c|c|c||c|c|c||c|c|c||c|c|c|}
\hline
&Angles&\multicolumn{6}{|c||}{100\%}\cr
\cline{2-8}
&Projections&\multicolumn{3}{|c||}{20}&\multicolumn{3}{|c||}{120}\cr
\cline{2-8}
&Resolution&82&164&328&82&164&328\cr
\hline
&$\rm it_{ALM}$&8&8&9&8&10&10\cr
SSSN&$\rm it_{SSSN}$&74&98&141&53&88&132\cr
$\eps_{\op{Opt}}=10^{-6}$&cg&892&1522&2761&522&1087&2097\cr
&time&1.1&8.5&67.1&1.1&11.6&88.3\cr
\hline
&$\rm it_{ALM}$&15&13&18&14&18&19\cr
SSSN&$\rm it_{SSSN}$&89&113&175&70&120&171\cr
$\eps_{\op{Opt}}=10^{-9}$&cg&1123&1839&3726&721&1531&2752\cr
&time&1.4&10.0&88.5&1.5&16.1&116.9\cr
\hline
&it&1305&2382&3725&562&836&1332\cr
CP&cg&5235&9163&12268&1951&2513&3884\cr
$\eps_{\op{Opt}}^{\op{CP}}=10^{-6}$&time&1.8&11.5&159.1&2.0&16.1&102.7\cr
\hline
&it&460&702&829&345&497&602\cr
ADMM&cg&7533&17045&24853&3214&6175&8711\cr
$\eps_{\op{Opt}}^{\op{ADMM}}=10^{-6}$&time&2.0&16.7&578.5&2.8&31.1&178.5\cr
\hline
\end{tabular}
\caption{Test setting I (CT): Performance metrics of the considered methods. Here, ``SSSN'' stands for our \ssstar Newton approach, i.e., Algorithm~\ref{AlgALM} in conjunction with Algorithm~\ref{AlgSubProbl}, and ``CP'' and ``ADMM'' for Algorithms~\ref{AlgChPo} and \ref{AlgADMM}, respectively. The number of iterations required by Algorithm~\ref{AlgALM} is denoted by $it_{\rm ALM}$, whereas $it_{\rm SSSN}$ is the total number of iterations performed by Algorithm~\ref{AlgSubProbl}. For ``CP'' and ``ADMM'', {\em it} denotes the number of iterations in the respective algorithm. For all three methods, {\em cg} and {\em time} report the total number of CG iterations and CPU time measured in seconds.}
\label{TabWalnut_Comp_100}
\end{table}

\begin{figure}[ht!]
    \centering
    \includegraphics[width=\textwidth]{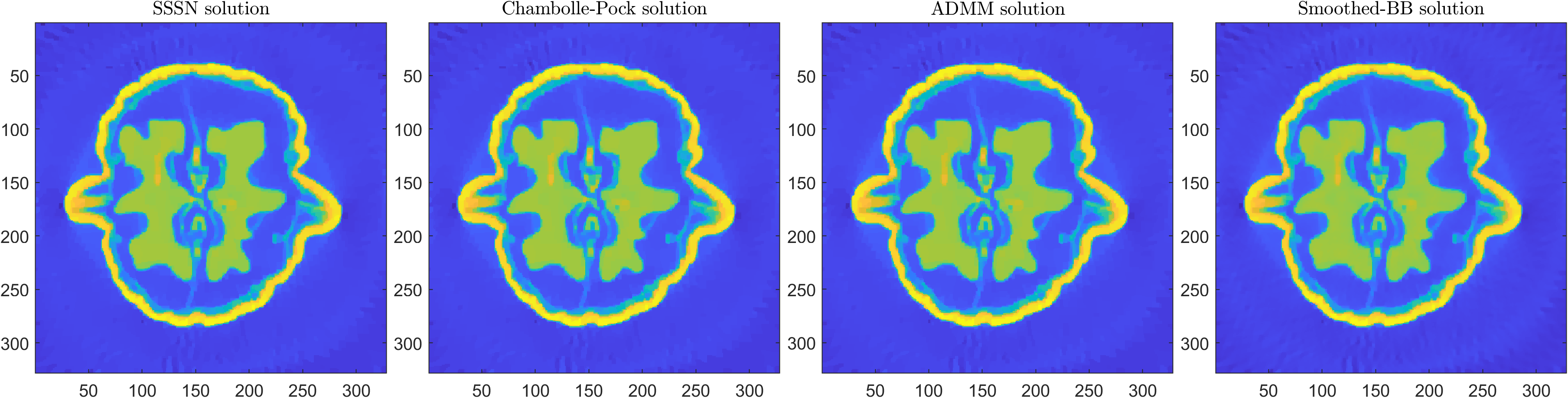}
    \caption{Test setting I (CT): Comparison of reconstructions for a representative test configuration ($328 \times 328$ pixels, $120$ projections, $100\%$ angle). Here, ``SSSN'' stands for our \ssstar Newton approach, i.e., Algorithm~\ref{AlgALM}, while ``Smoothed-BB'' stands for the smoothed Barzilai-Borwein approach (with $\eps = 10^{-8}$) defined in \eqref{method_BB}.}
    \label{fig_Walnut_Reconstructions}
\end{figure}

\begin{figure}[ht!]
    \centering
    \includegraphics[width=\textwidth]{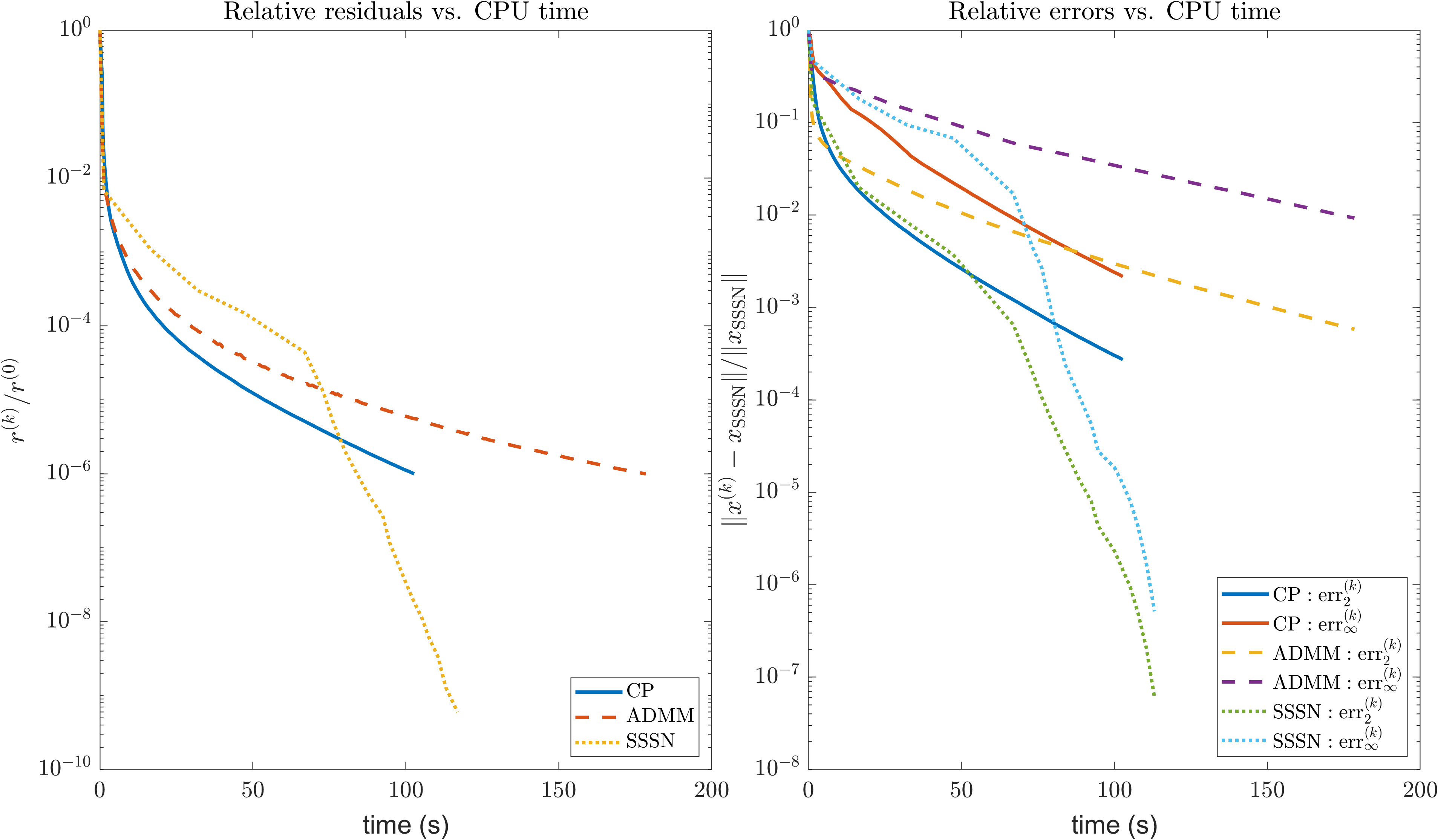}
    \caption{Test setting I (CT): Comparison of relative residuals and relative errors, both vs.\ CPU time, for a representative test configuration ($328 \times 328$ pixels, $120$ projections, $100\%$ angle). Here, ``SSSN'' stands for our \ssstar Newton approach, i.e., Algorithm~\ref{AlgALM}, and ``CP'' stands for the Chambolle-Pock method, i.e., Algorithm~\ref{AlgChPo}.}
    \label{fig_Walnut_Residuals}
\end{figure}

In all considered test configurations (i.e., different spatial resolutions, numbers of parallel rays, limited-angle settings),  the proposed Algorithm~\ref{AlgALM} (with stopping tolerance $10^{-6}$) is faster than Algorithm~\ref{AlgChPo} (Chambolle-Pock), whereas Algorithm~\ref{AlgADMM} (ADMM)  required  considerably more time. Even with stopping tolerance $10^{-9}$, Algorithm \ref{AlgALM} is in most cases faster than the other two methods with stopping tolerance $10^{-6}$. In Table~\ref{TabWalnut_Comp_100}, we report the number of iterations, CG iterations, and CPU time required for each method to reach the prescribed stopping tolerance in  case of $100\%$ angles. The results for other percentages of angles are given in Appendix~\ref{appendix}.
The performance of the first-order comparison method \eqref{method_BB} based on the smooth approximation \eqref{approx} of the TV penalty term is discussed separately below.

Figure~\ref{fig_Walnut_Reconstructions} presents the reconstructions obtained with the different algorithms for a representative test configuration ($328\times 328$ pixels, $120$ projections, $100\%$ angle), while Figure~\ref{fig_Walnut_Residuals} depicts the evolution, over computation time, of the corresponding relative residual $\ee r{k}/\ee r0$ and the relative errors $\ee {\rm err}{k}_2$ and $\ee {\rm err}{k}_\infty$. The results for the other test settings, which are generally very similar, are given in Appendix~\ref{appendix}. Note that the comparison over computation time as opposed to over iteration number was chosen to allow for an objective comparison of the considered algorithms. Throughout all tests, within the first iterations, both Algorithm~\ref{AlgChPo} and Algorithm~\ref{AlgADMM} reduce the residual $\ee rk$ faster than Algorithm~\ref{AlgALM}, but then slow down and are outperformed by Algorithm~\ref{AlgALM}. Similarly 
it appears that during the first iterations, the relative errors $\ee {\rm err}k_2$ decreases equally fast for Algorithm~\ref{AlgALM} and Algorithm~\ref{AlgChPo}, but again with increasing CPU time Algorithm~\ref{AlgALM} excels Algorithm~\ref{AlgChPo}. We can see that, although the residuals $\ee rk$ differs by a factor of magnitude $10$, the distance to the solution is nearly the same. Furthermore, as before, ADMM performs worse than the other two methods wrt.\ the distance to the solution.

Concerning the performance of the first-order comparison method \eqref{method_BB}, we note the following: First, we found that the choice of the smoothing parameter $\epsilon$ used in the approximation of the TV penalty strongly affected both the stability, efficiency, and accuracy of the method. In particular, with a large $\epsilon$, the method does indeed converge very quickly (thanks to the Barzilai-Borwein stepsize), in some cases even faster than our second-order approach, but then also results in very large errors. E.g., in Figure~\ref{fig_Walnut_Reconstructions}, although the depicted reconstruction obtained with \eqref{method_BB} looks good visually, we have $\ee rk = 0.0358$, $\ee {\rm err}k_2 = 0.011$, and $\ee {\rm err}k_\infty = 0.0272$, which is orders of magnitude above the values obtained with the other reconstruction algorithms. On the other hand, for small smoothing parameters $\epsilon$, method \eqref{method_BB} typically becomes instable and fails to converge. This issue is further compounded by the difficulty to determine an optimal smoothing parameter $\epsilon$, which in our tests appears to depend both on the resolution, number of projections, and limited-angle setting. These observations are consistent throughout all our tests, and in a sense are to be expected, given that \eqref{method_BB} solves only a smooth approximation of the actual, non-smooth minimization problem \eqref{EqLSProbl}. Due to these issues, we ultimately decided not to include \eqref{method_BB} in our residual/error plots.

\begin{table}[ht!]
\centering
\begin{tabular}{|c|c|c|c|c|c|c|c|c|}
\hline
$k$&$\ee\sigma k$&$\ee r{k+1}/\ee r0$&$\ee {\rm err}{k+1}_2$&$\ee {\rm err}{k+1}_\infty$&time&it&cg it&$\vert I(\ee z{k+1})\vert$\cr
\hline
0&7.8e+02&6.8e-03&1.7e-01&4.7e-01&1.29&1&9&168609\cr
1&7.8e+02&1.1e-03&2.0e-02&1.8e-01&15.26&20&392&106697\cr
2&1.4e+03&3.0e-04&8.4e-03&9.4e-02&15.54&24&369&94737\cr
3&1.4e+03&1.5e-04&3.7e-03&6.8e-02&15.20&25&360&79189\cr
4&1.4e+03&4.4e-05&6.5e-04&1.7e-02&19.68&30&460&65314\cr
5&1.4e+03&1.3e-05&2.1e-04&4.7e-03&6.03&10&142&56628\cr
6&1.4e+03&4.9e-06&1.1e-04&2.7e-03&3.29&5&80&54842\cr
7&1.4e+03&1.9e-06&5.2e-05&6.7e-04&4.03&6&98&52769\cr
8&2.5e+03&1.0e-06&2.9e-05&2.4e-04&3.64&5&86&52480\cr
9&4.0e+03&4.8e-07&1.4e-05&1.1e-04&4.30&6&101&52021\cr
10&4.0e+03&2.7e-07&7.9e-06&5.4e-05&4.23&6&104&51634\cr
11&6.2e+03&1.2e-07&4.3e-06&2.8e-05&2.21&3&51&51458\cr
12&6.2e+03&3.3e-08&2.3e-06&1.8e-05&5.36&8&122&51046\cr
13&6.2e+03&2.0e-08&1.6e-06&1.3e-05&2.25&3&55&51044\cr
14&9.3e+03&1.1e-08&9.3e-07&7.9e-06&2.94&4&65&50996\cr
15&1.4e+04&6.0e-09&5.0e-07&4.3e-06&2.41&3&53&50915\cr
16&1.9e+04&3.4e-09&2.0e-07&1.7e-06&2.84&4&58&50812\cr
17&1.9e+04&1.3e-09&6.0e-08&5.1e-07&2.62&3&65&50742\cr
18&1.9e+04&6.0e-10&-&-&3.81&5&82&50649\cr
\hline
\end{tabular}
\caption{Test setting I (CT): Performance metrics of Algorithm~\ref{AlgALM} for a representative test configuration ($328 \times 328$ pixels, $120$ projections, $100\%$ angle). Here, {\em time} is the CPU time required by Algorithm~\ref{AlgSubProbl} to approximately solve \eqref{EqApprMinLaq}, {\em it} is the number of iterations performed in each call of  Algorithm~\ref{AlgSubProbl}, and {\em cg it} is the total number of CG iterations required in Steps~1 and 2 of Algorithm~\ref{AlgSubProbl}. Furthermore, $\vert I(\ee z{k+1}) \vert$ is the cardinality of the index set $I(\ee z{k+1})$, i.e., the number of nonzero components~of~$\ee z{k+1}$.}
\label{TabWalnut328_120}
\end{table}

In Table~\ref{TabWalnut328_120}, we provide detailed information about the individual iterations performed by Algorithm~\ref{AlgALM} for the same test configuration considered above (i.e., $328 \times 328$ pixels, $120$ projections, $100\%$ angle). We can see that most of the computation time is required in the first iterations of Algorithm~\ref{AlgALM}, while for later iterations, the numerical effort for evaluating \eqref{EqApprMinLaq} is much smaller. We observed that in most cases, in the last iteration of Algorithm~\ref{AlgSubProbl} we obtained a considerable reduction of the residual by a factor of up to $0.001$, indicating that we reached the region of superlinear convergence. Furthermore, from the values of $\vert I(\ee z{k+1}) \vert$ we can see that for the last iterations, when we are already close to a solution, there is still a considerable change in these index sets. We conjecture that this phenomenon is due to the dual degeneracy of the problem; cf.~the last paragraph of Section~\ref{subsect_implementation}.

Finally, note that throughout the different test settings, the iteration numbers stay reasonably constant, indicating that our Algorithm~\ref{AlgALM} is in fact discretization invariant.

\begin{remark}
When omitting Step~1 in Algorithm~\ref{AlgSubProbl}, i.e., setting $\ee{\hat x}j = \ee xj$ and fixing the regularization parameter $\ee \rho j=0$, Algorithm~\ref{AlgSubProbl} can be considered as a standard inexact semismooth Newton method for minimizing the function $\vartheta$ given by \eqref{EqTheta}. However, this method appears to be inefficient and requires more than four times more computation time than our implementation. Furthermore, it has big troubles to come close to a solution of the problem by computing too large search directions $\ee{\triangle}j$, which result in small step sizes during the line search.
\end{remark}

\subsection{Numerical results for setting II: PAT}

Next, we present the numerical results in the PAT test setting described in Section~\ref{subsect_setting_PAT}. Note first that in our implementation, we have two separate routines for evaluating $Ax$ and $\tilde A^*y$, where $\tilde A^*$ is only an approximation for $A^*$. In fact, for random $x$ and $b$,
    \begin{equation*}
        \abs{\skalp{Ax,b}-\skalp{\tilde A^*b,x}}/\abs{\norm{Ax}\norm{b}} \approx 10^{-8} \,,
    \end{equation*}
which is too inaccurate to allow for high-precision computations.
In particular, during our numerical tests, we observed that in all cases the CG method worked, but Algorithm~\ref{AlgSubProbl} broke down with difficulties during the linesearch when using stopping tolerances $\eps_{\op{Opt}}<10^{-5}$. Similarly, in Algorithms~\ref{AlgChPo} and \ref{AlgADMM} the residual could not be further reduced after several iterations impeding proper termination. To overcome this problem, we computed in advance, in a very time consuming process, for each of the three  sensor arrangements the matrix $A$  for the low resolution of $128\times 128$ pixels, and ran the three Algorithms~\ref{AlgALM}, \ref{AlgChPo}, and \ref{AlgADMM} with stopping tolerances $\eps_{\op{Opt}}=\eps_{\op{Opt}}^{\op{CP}}=\eps_{\op{Opt}}^{\op{ADMM}}=10^{-3}$, respectively. The solution $\xb_{\op{ALM}}$ needed to compute the relative errors in \eqref{EqRelErr} was computed with $\eps_{\op{Opt}}=10^{-6}$.
Figure~\ref{fig_PAT_Reconstructions} presents the reconstructions obtained in a representative test configuration (two-sided), while Figure~\ref{fig_PAT_Residuals} depicts the evolution, over computation time, of the corresponding relative residual $\ee r{k}/\ee r0$ and the relative errors $\ee {\rm err}{k}_2$ and $\ee {\rm err}{k}_\infty$. The results for the other test settings and performance metrics are given in Appendix~\ref{appendix}.

\begin{figure}[ht!]
    \centering
    \includegraphics[width=\textwidth]{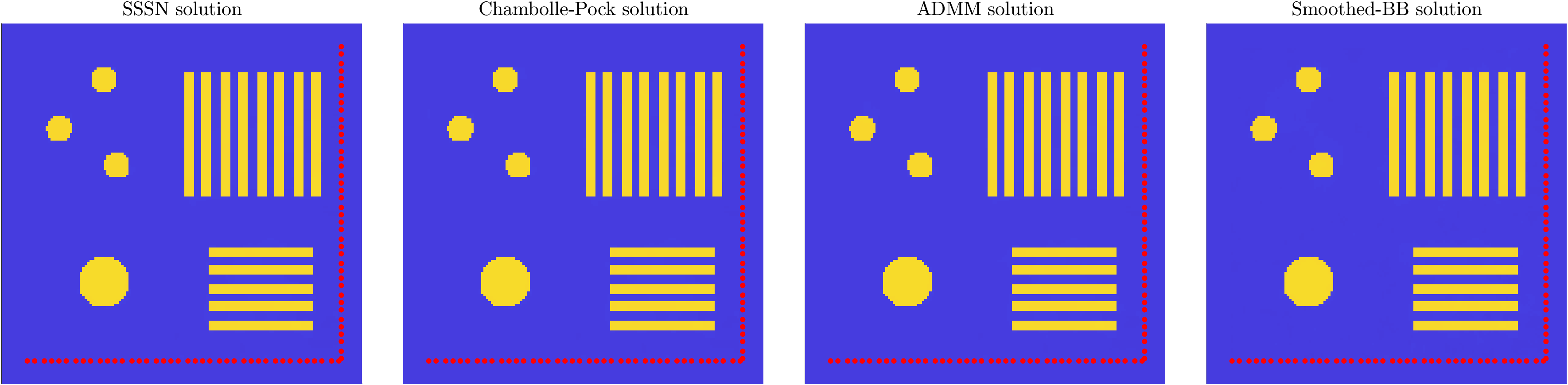}
    \caption{Test setting II (PAT): Comparison of reconstructions for a representative test configuration ($128 \times 128$ pixels, two-sided sensor layout). Here, ``SSSN'' stands for our \ssstar Newton approach, i.e., Algorithm~\ref{AlgALM}, while ``Smoothed-BB'' stands for the smoothed Barzilai-Borwein approach (with $\eps = 10^{-6}$) defined in \eqref{method_BB}.}
    \label{fig_PAT_Reconstructions}
\end{figure}

\begin{figure}[ht!]
    \centering
    \includegraphics[width=\textwidth]{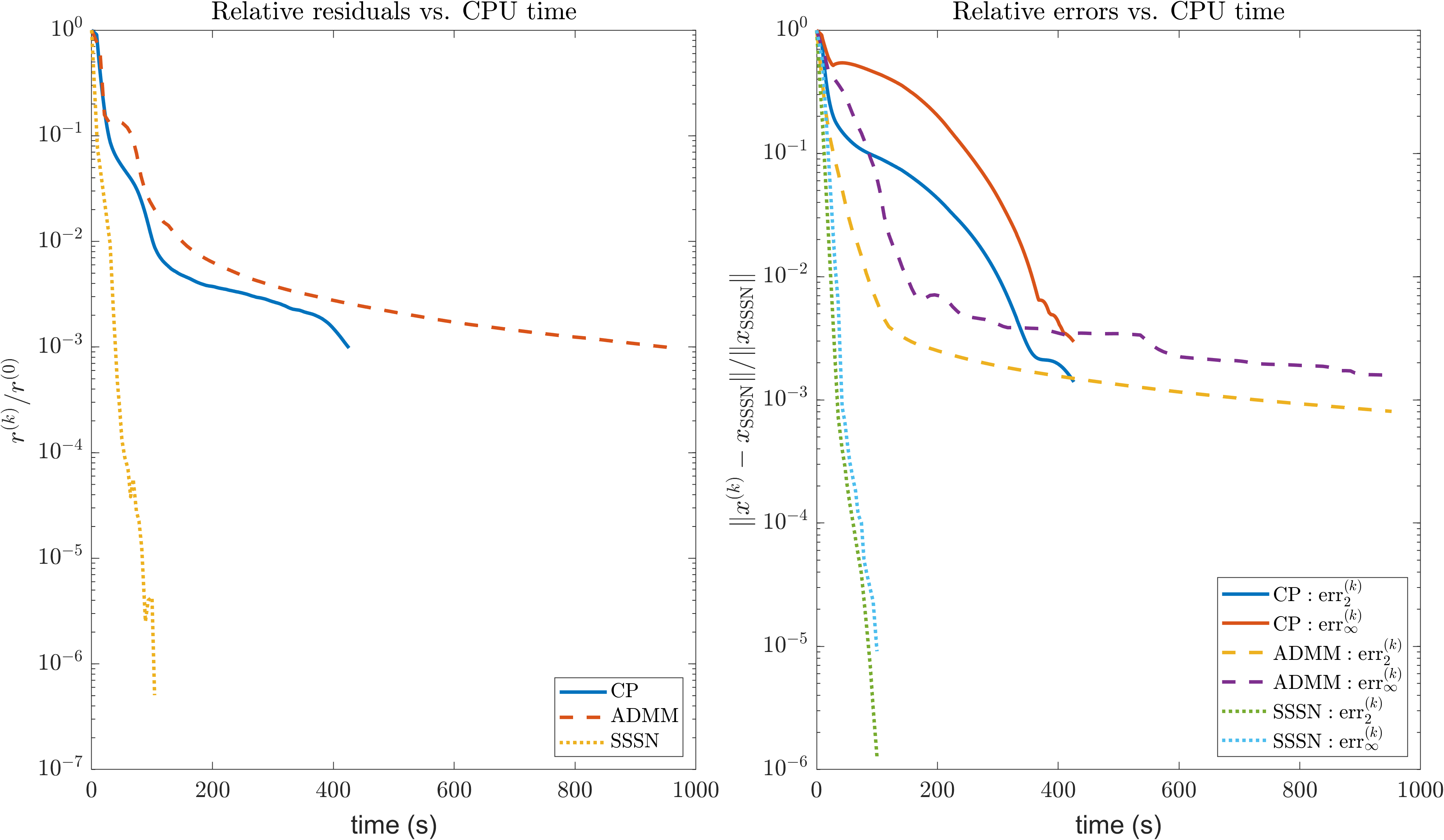}
    \caption{Test setting II (PAT): Comparison of relative residuals and relative errors, both vs.\ CPU time, for a representative test configuration ($128 \times 128$ pixels, two-sided sensor layout). Here, ``SSSN'' stands for our \ssstar Newton approach, i.e., Algorithm~\ref{AlgALM}, and ``CP'' stands for the Chambolle-Pock method, i.e., Algorithm~\ref{AlgChPo}. .}
    \label{fig_PAT_Residuals}
\end{figure}

\begin{figure}[ht!]
    \centering
    \includegraphics[width=\textwidth]{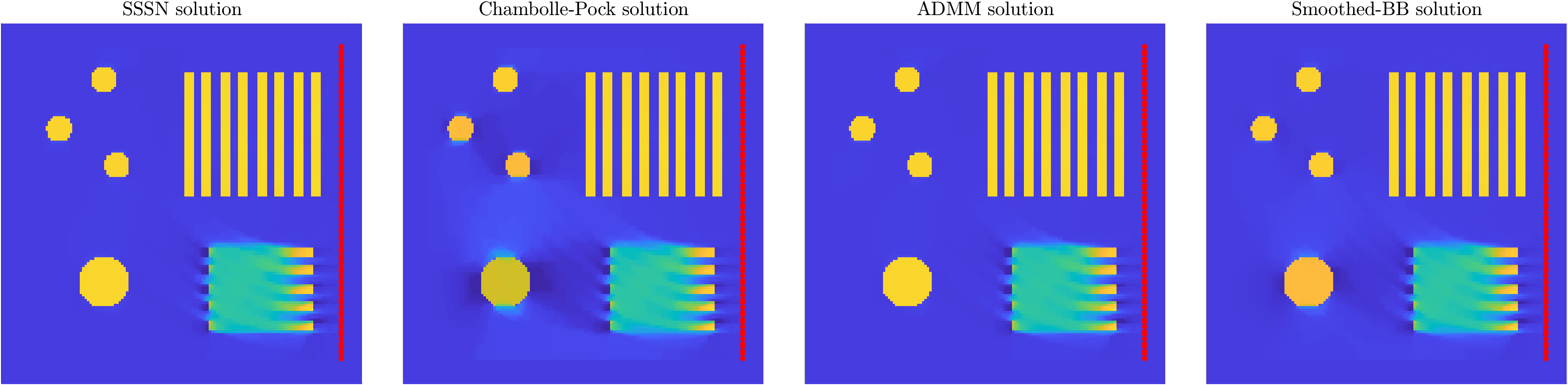}
    \caption{Test setting II (PAT): Comparison of reconstructions for a different test configuration ($128 \times 128$ pixels, one-sided sensor layout). Here, ``SSSN'' stands for our \ssstar Newton approach, i.e., Algorithm~\ref{AlgALM}, while ``Smoothed-BB'' stands for the smoothed Barzilai-Borwein approach (with $\eps = 10^{-6}$) defined in \eqref{method_BB}.}
    \label{fig_PAT_onesided}
\end{figure}

It appears that in the one-sided sensor case, the solution is not unique, and thus the reconstructions contain artifacts; see Figure~\ref{fig_PAT_onesided}. Furthermore, in all PAT tests, Algorithm~\ref{AlgALM} appears to be superior to the other two methods throughout the iteration: Algorithm~\ref{AlgALM} reduces the residual $\ee rk$ much faster than the Chambolle-Pock and ADMM method, respectively, and also the deviation to the solution (in case when it is unique) decreases much faster. Moreover, we observed that a stopping tolerance of $\eps_{\op{Opt}}=10^{-4}$ resulted in a relative error of less then $10^{-3}$ with respect to both the $\ell_2$ norm and the $\ell_\infty$ norm. Hence, in a separate test series with a higher resolution of $512 \times 512$ pixels, where we (are forced to) use the matrix-free implementation of $A$, we decided to run Algorithm~\ref{AlgALM} with this stopping tolerance. Figure~\ref{fig_PAT_High_Res} presents both the obtained reconstruction and the corresponding relative residuals and errors, in one of these high-resolution cases (two-sided). Note that due to the large-scale nature of the PAT problem at this resolution, only Algorithm~\ref{AlgALM} was able to converge to a solution within an acceptable time-frame, and thus only those results are presented in the figure.  

Again, note that throughout the different test settings, the iteration numbers stay reasonably constant, indicating that our Algorithm~\ref{AlgALM} is in fact discretization invariant.

\begin{figure}[ht!]
    \centering
    \includegraphics[width=\textwidth, trim = {1cm 4.2cm 0cm 4.3cm}, clip=true]{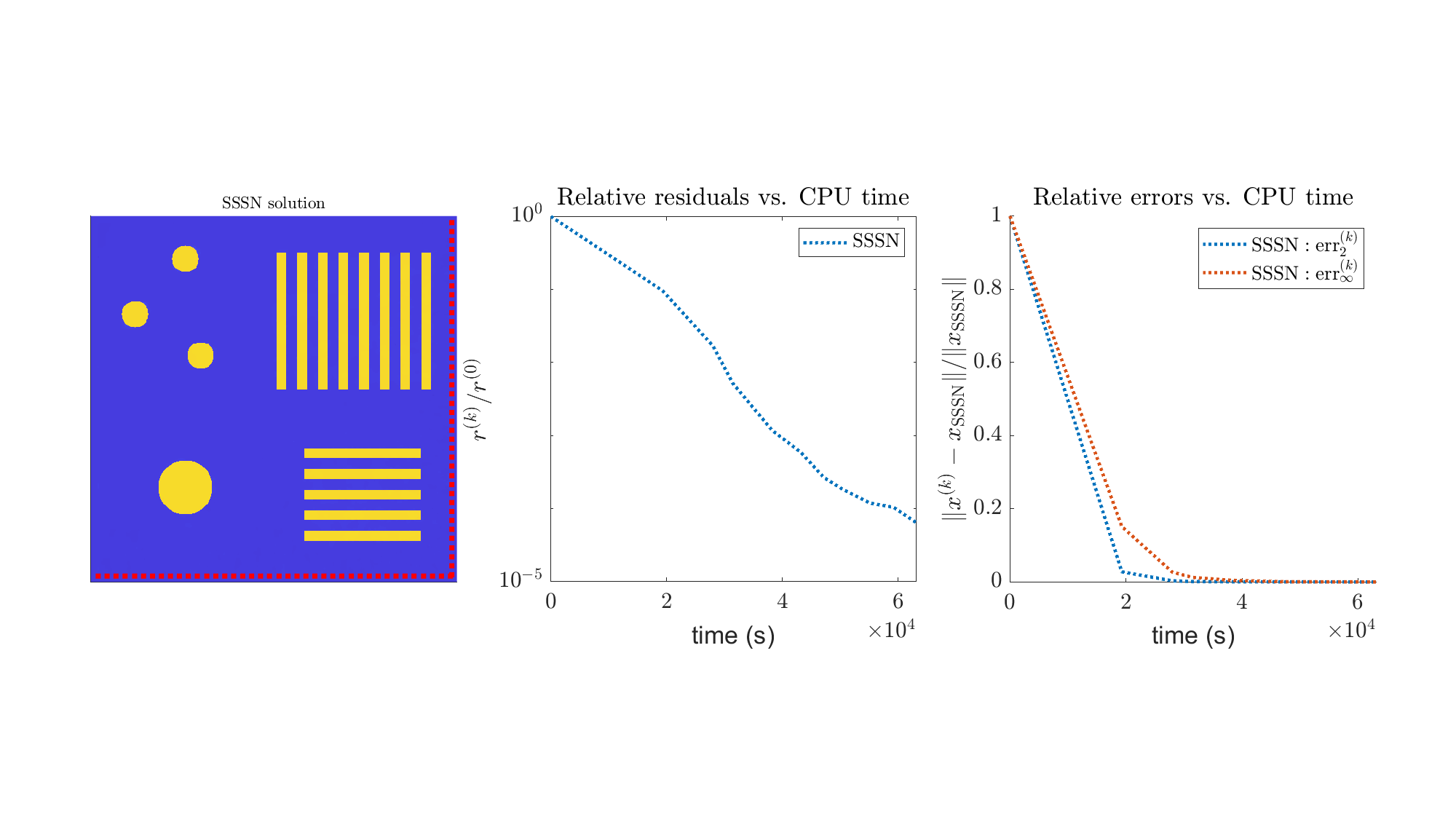}
    \caption{Test setting II (PAT): Reconstruction and comparison of relative residuals and relative errors, both vs.\ CPU time, for a high-resolution test configuration ($512 \times 512$ pixels, two-sided sensor layout). Here, ``SSSN'' stands for our \ssstar Newton approach, i.e., Algorithm~\ref{AlgALM}. The minima of the depicted error curves are: $9.6\cdot 10^{-5}$ for the relative residual, $0.0002$ for $\ee {\rm err}k_2$, and $0.0028$ for $\ee {\rm err}k_\infty$.}
    \label{fig_PAT_High_Res}
\end{figure}

\section{Conclusion}\label{sect_conclusion}

In this paper, we considered the efficient numerical minimization of Tikhonov functionals resulting from TV regularization of linear inverse problems. For this, we proposed a minimization approach based on the \ssstar Newton method, which uses the novel concept of graphical derivatives to generalize the classical Newton method to non-smooth, set-valued mappings. The proposed method is applicable to large-scale inverse problems, and is supported by strong convergence guarantees, including locally superlinear convergence. Numerical experiments on two large-scale tomographic imaging problems from X-ray CT and PAT demonstrated that our proposed approach is competitive with, and in terms of convergence speed and optimization accuracy clearly outperforms, existing state-of-the-art methods for TV regularization.

\section{Acknowledgments \& Support}

This research was funded in part by the Austrian Science Fund (FWF) SFB 10.55776/F68 ``Tomography Across the Scales'', project F6805-N36 (Tomography in Astronomy). For open access purposes, the authors have applied a CC BY public copyright license to any author-accepted manuscript version arising from this submission. This project has received funding from the European Research Council (ERC) under the European Union’s Horizon 2020 research and innovation programme (grant agreement No 101001417 - QUANTOM) and the Research Council of Finland (Flagship of Advanced Mathematics for Sensing Imaging and Modelling grant 358944).

\bibliographystyle{plain}
{\footnotesize
\bibliography{mybib,stefan}
}

\appendix

\section{Supplemental figures of numerical results}\label{appendix}

\begin{figure}[ht!]
    \centering
    \includegraphics[width=\textwidth]{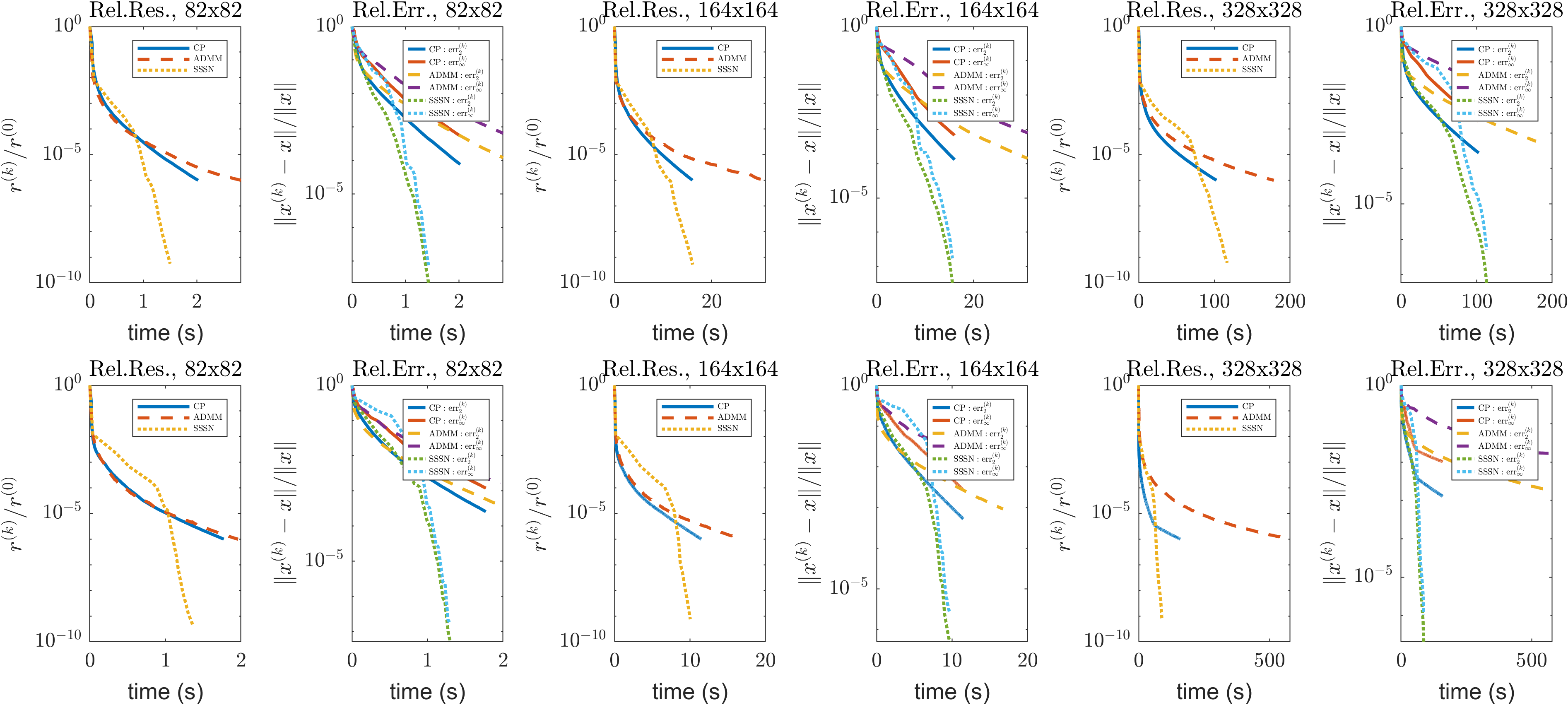}
    \caption{Test setting I (CT): Comparison of relative residuals and relative errors, both vs.\ CPU time, for different test configurations in the $100\%$ limited-angle case: pixel dimensions $82 \times 82$, $164 \times 164$, $328 \times 328$, and both $120$ (top row) and $20$ (bottom row) projections. Here, ``SSSN'' stands for our \ssstar Newton approach, i.e., Algorithm~\ref{AlgALM}, and ``CP'' stands for the Chambolle-Pock method, i.e., Algorithm~\ref{AlgChPo}.}
    \label{fig_A1}
\end{figure}

\begin{figure}[ht!]
    \centering
    \includegraphics[width=\textwidth]{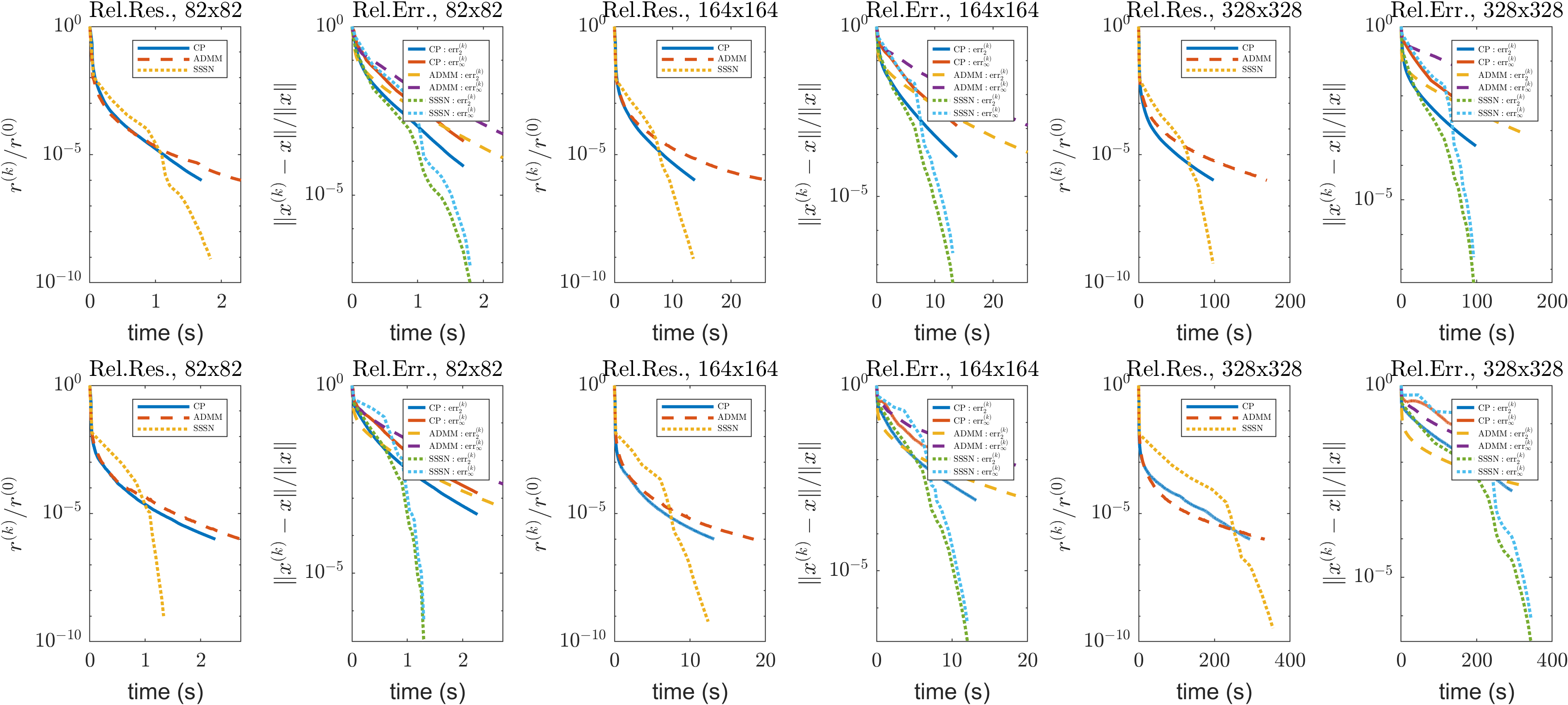}
    \caption{Test setting I (CT): Comparison of relative residuals and relative errors, both vs.\ CPU time, for different test configurations in the $75\%$ limited-angle case: pixel dimensions $82 \times 82$, $164 \times 164$, $328 \times 328$, and both $120$ (top row) and $20$ (bottom row) projections. Here, ``SSSN'' stands for our \ssstar Newton approach, i.e., Algorithm~\ref{AlgALM}, and ``CP'' stands for the Chambolle-Pock method, i.e., Algorithm~\ref{AlgChPo}.}
    \label{fig_A2}
\end{figure}

\begin{figure}[ht!]
    \centering
    \includegraphics[width=\textwidth]{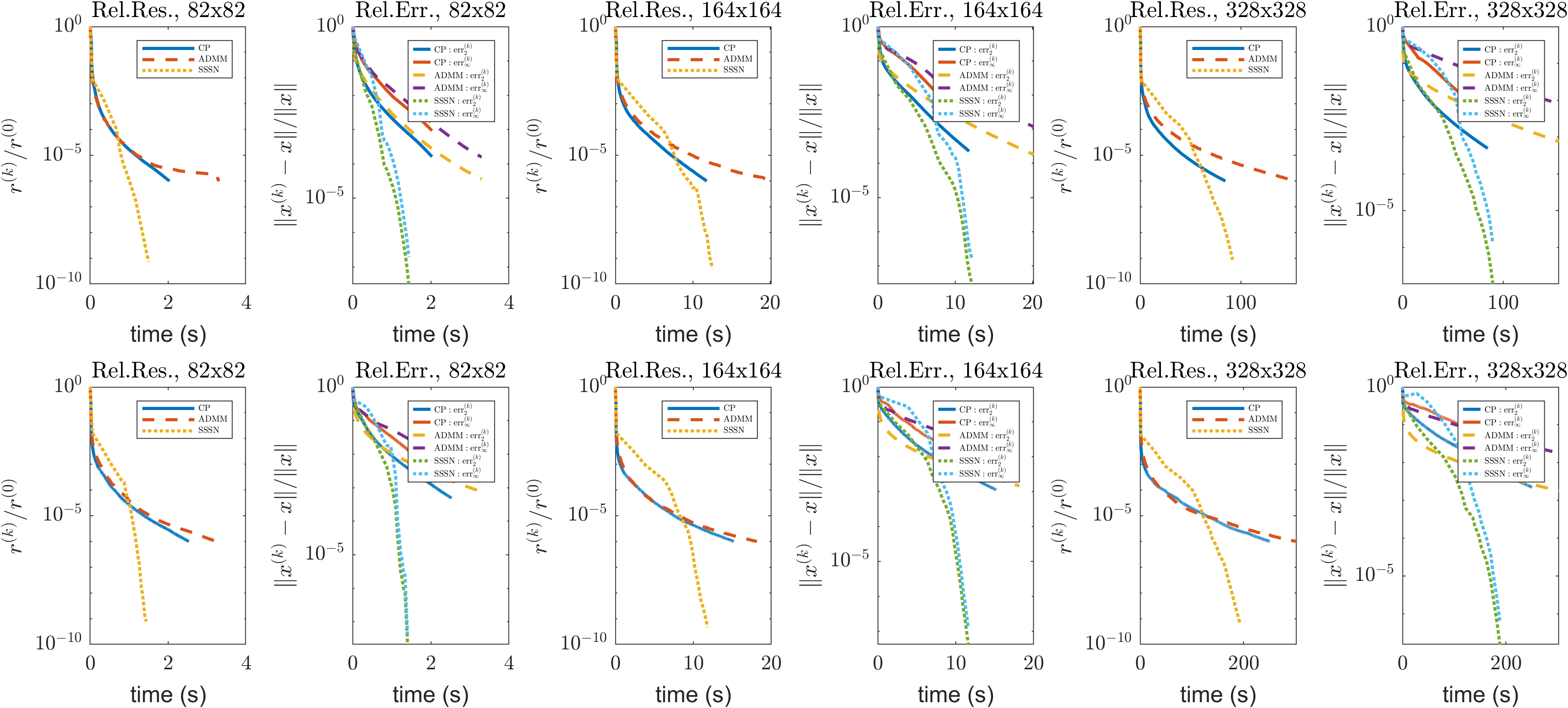}
    \caption{Test setting I (CT): Comparison of relative residuals and relative errors, both vs.\ CPU time, for different test configurations in the $50\%$ limited-angle case: pixel dimensions $82 \times 82$, $164 \times 164$, $328 \times 328$, and both $120$ (top row) and $20$ (bottom row) projections. Here, ``SSSN'' stands for our \ssstar Newton approach, i.e., Algorithm~\ref{AlgALM}, and ``CP'' stands for the Chambolle-Pock method, i.e., Algorithm~\ref{AlgChPo}.}
    \label{fig_A3}
\end{figure}

\begin{figure}[ht!]
    \centering
    \includegraphics[width=\textwidth]{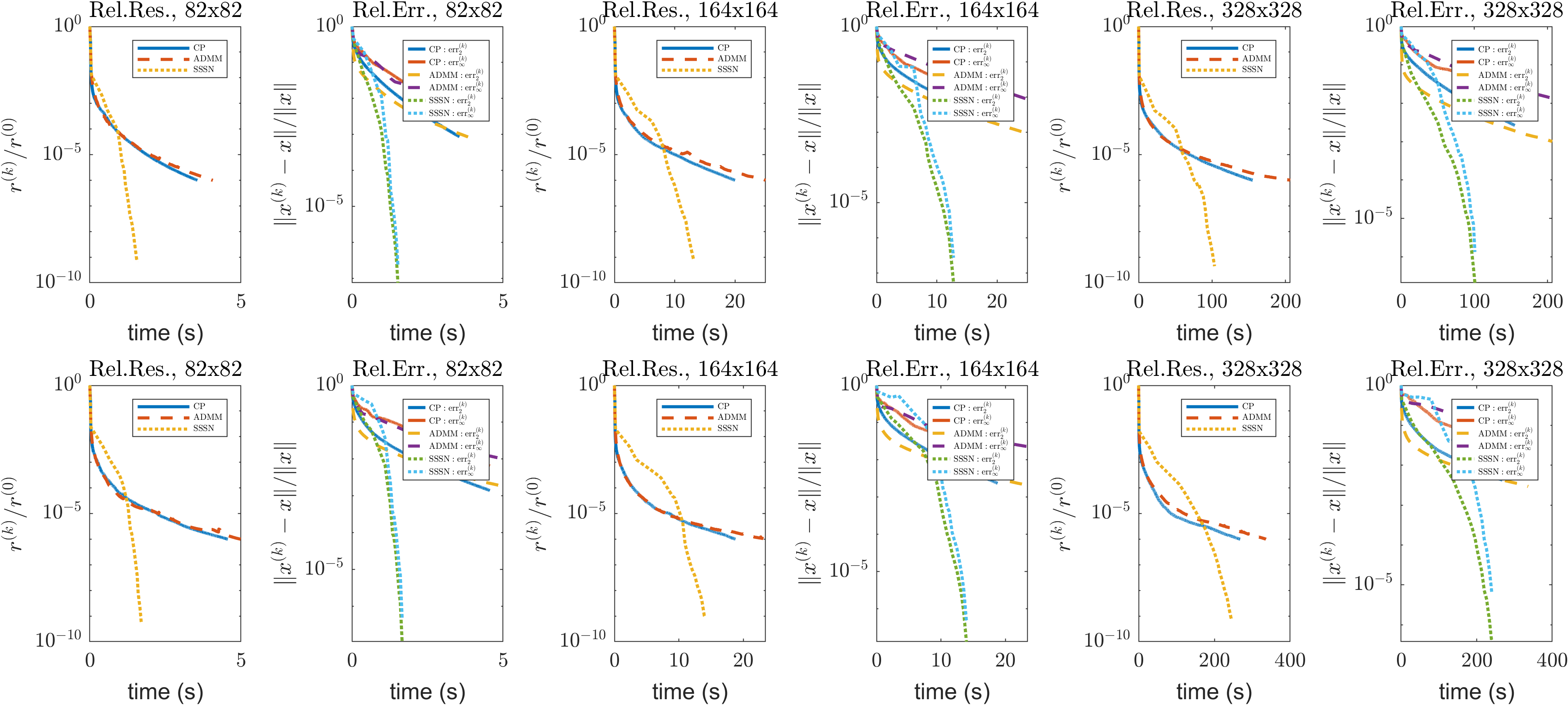}
    \caption{Test setting I (CT): Comparison of relative residuals and relative errors, both vs.\ CPU time, for different test configurations in the $25\%$ limited-angle case: pixel dimensions $82 \times 82$, $164 \times 164$, $328 \times 328$, and both $120$ (top row) and $20$ (bottom row) projections. Here, ``SSSN'' stands for our \ssstar Newton approach, i.e., Algorithm~\ref{AlgALM}, and ``CP'' stands for the Chambolle-Pock method, i.e., Algorithm~\ref{AlgChPo}.}
    \label{fig_A4}
\end{figure}

\begin{table}[ht!]
\centering
\begin{tabular}{|c|c||c|c|c||c|c|c||c|c|c||c|c|c|}
\hline
&Angles&\multicolumn{6}{|c||}{75\%}\cr
\cline{2-8}
&Projections&\multicolumn{3}{|c||}{20}&\multicolumn{3}{|c||}{120}\cr
\cline{2-8}
&Resolution&82&164&328&82&164&328\cr
\hline
&$\rm it_{ALM}$&9&9&9&8&10&10\cr
SSSN&$\rm it_{SSSN}$&73&95&158&63&79&125\cr
$\eps_{\op{Opt}}=10^{-6}$&cg&1062&1714&3340&687&1069&2155\cr
&time&1.2&8.6&259.4&1.2&9.7&79.1\cr
\hline
&$\rm it_{ALM}$&12&19&19&18&17&17\cr
SSSN&$\rm it_{SSSN}$&81&128&197&93&107&152\cr
$\eps_{\op{Opt}}=10^{-9}$&cg&1204&2654&4654&1065&1520&2696\cr
&time&1.3&12.4&354.3&1.8&13.5&98.3\cr
\hline
&it&1680&2847&4718&573&938&1603\cr
CP&cg&6772&11411&16589&2173&2882&4845\cr
$\eps_{\op{Opt}}^{\op{CP}}=10^{-6}$&time&2.3&13.2&294.1&1.7&13.8&99.0\cr
\hline
&it&502&712&888&336&533&647\cr
ADMM&cg&11138&21002&26623&3531&6957&10679\cr
$\eps_{\op{Opt}}^{\op{ADMM}}=10^{-6}$&time&2.7&18.4&332.8&2.3&25.9&169.2\cr
\hline
\end{tabular}
\caption{Test setting I (CT): Performance metrics of the considered methods in the $75\%$ limited-angle case.}
\label{TabWalnut_Comp_75}
\end{table}

\begin{table}[ht!]
\centering
\begin{tabular}{|c|c||c|c|c||c|c|c||c|c|c||c|c|c|}
\hline
&Angles&\multicolumn{6}{|c||}{50\%}\cr
\cline{2-8}
&Projections&\multicolumn{3}{|c||}{20}&\multicolumn{3}{|c||}{120}\cr
\cline{2-8}
&Resolution&82&164&328&82&164&328\cr
\hline
&$\rm it_{ALM}$&10&10&10&9&10&10\cr
SSSN&$\rm it_{SSSN}$&76&107&138&61&94&122\cr
$\eps_{\op{Opt}}=10^{-6}$&cg&1158&2094&3160&624&1238&2133\cr
&time&1.2&9.6&147.9&1.0&9.7&68.0\cr
\hline
&$\rm it_{ALM}$&15&18&19&17&18&19\cr
SSSN&$\rm it_{SSSN}$&88&128&169&86&122&160\cr
$\eps_{\op{Opt}}=10^{-9}$&cg&1344&2643&4222&958&1596&2905\cr
&time&1.4&11.8&192.7&1.5&12.5&91.8\cr
\hline
&it&2069&3707&5932&851&1140&1921\cr
CP&cg&8323&14844&22192&3264&3784&5644\cr
$\eps_{\op{Opt}}^{\op{CP}}=10^{-6}$&time&2.5&15.2&250.8&2.0&11.8&84.7\cr
\hline
&it&601&792&964&528&554&735\cr
ADMM&cg&14894&23741&28901&6863&8259&13889\cr
$\eps_{\op{Opt}}^{\op{ADMM}}=10^{-6}$&time&3.3&18.1&302.4&3.3&20.1&155.3\cr
\hline
\end{tabular}
\caption{Test setting I (CT): Performance metrics of the considered methods in the $50\%$ limited-angle case.}
\label{TabWalnut_Comp_50}
\end{table}

\begin{table}[ht!]
\centering
\begin{tabular}{|c|c||c|c|c||c|c|c||c|c|c||c|c|c|}
\hline
&Angles&\multicolumn{6}{|c||}{25\%}\cr
\cline{2-8}
&Projections&\multicolumn{3}{|c||}{20}&\multicolumn{3}{|c||}{120}\cr
\cline{2-8}
&Resolution&82&164&328&82&164&328\cr
\hline
&$\rm it_{ALM}$&10&10&12&8&10&12\cr
SSSN&$\rm it_{SSSN}$&84&119&159&67&98&148\cr
$\eps_{\op{Opt}}=10^{-6}$&cg&1666&2676&4434&1024&1745&3122\cr
&time&1.5&11.0&188.8&1.2&10.0&78.6\cr
\hline
&$\rm it_{ALM}$&14&19&21&16&18&21\cr
SSSN&$\rm it_{SSSN}$&96&150&201&86&125&189\cr
$\eps_{\op{Opt}}=10^{-9}$&cg&1989&3538&5796&1399&2342&4108\cr
&time&1.7&14.1&245.2&1.6&13.1&103.3\cr
\hline
&it&3506&5452&7714&2065&3324&5274\cr
CP&cg&17201&22393&30987&8602&12624&16586\cr
$\eps_{\op{Opt}}^{\op{CP}}=10^{-6}$&time&4.6&18.8&268.1&3.6&20.0&155.2\cr
\hline
&it&865&1210&1339&436&654&984\cr
ADMM&cg&25940&36290&40158&12828&19606&29507\cr
$\eps_{\op{Opt}}^{\op{ADMM}}=10^{-6}$&time&5.0&23.4&336.2&4.1&25.1&206.3\cr
\hline
\end{tabular}
\caption{Test setting I (CT): Performance metrics of the considered methods in the $25\%$ limited-angle case.}
\label{TabWalnut_Comp_25}
\end{table}

\begin{figure}[ht!]
    \centering
    \includegraphics[width=\textwidth]{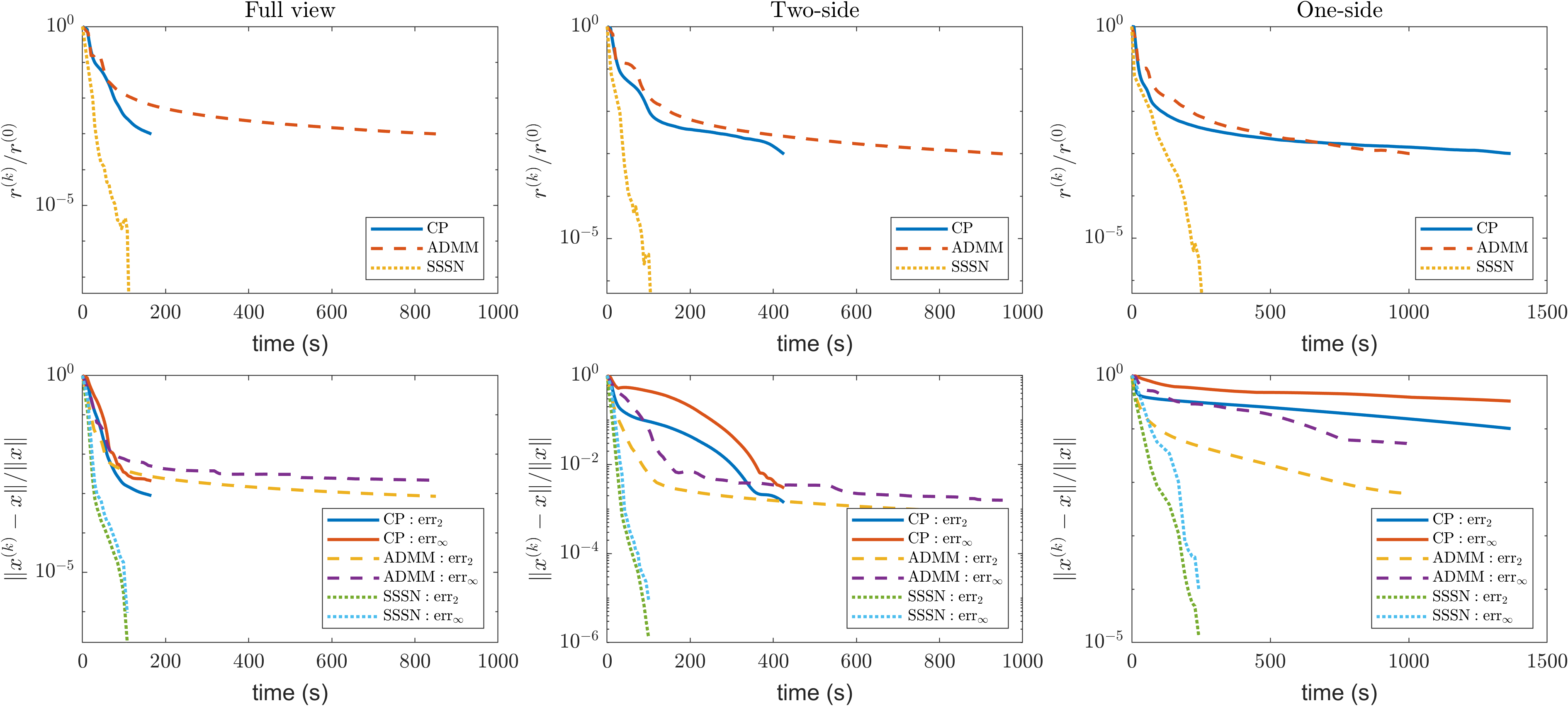}
    \caption{Test setting II (PAT): Comparison of relative residuals (top row) and relative errors (bottom row), both vs.\ CPU time, for different test configurations in the $128 \times 128$ pixel dimension case: Full-view (left), two-sided (middle), and one-sided (right) sensor layout. Here, ``SSSN'' stands for our \ssstar Newton approach, i.e., Algorithm~\ref{AlgALM}, and ``CP'' stands for the Chambolle-Pock method, i.e., Algorithm~\ref{AlgChPo}.}
    \label{fig_B1}
\end{figure}

\begin{figure}[ht!]
    \centering
    \includegraphics[width=\textwidth]{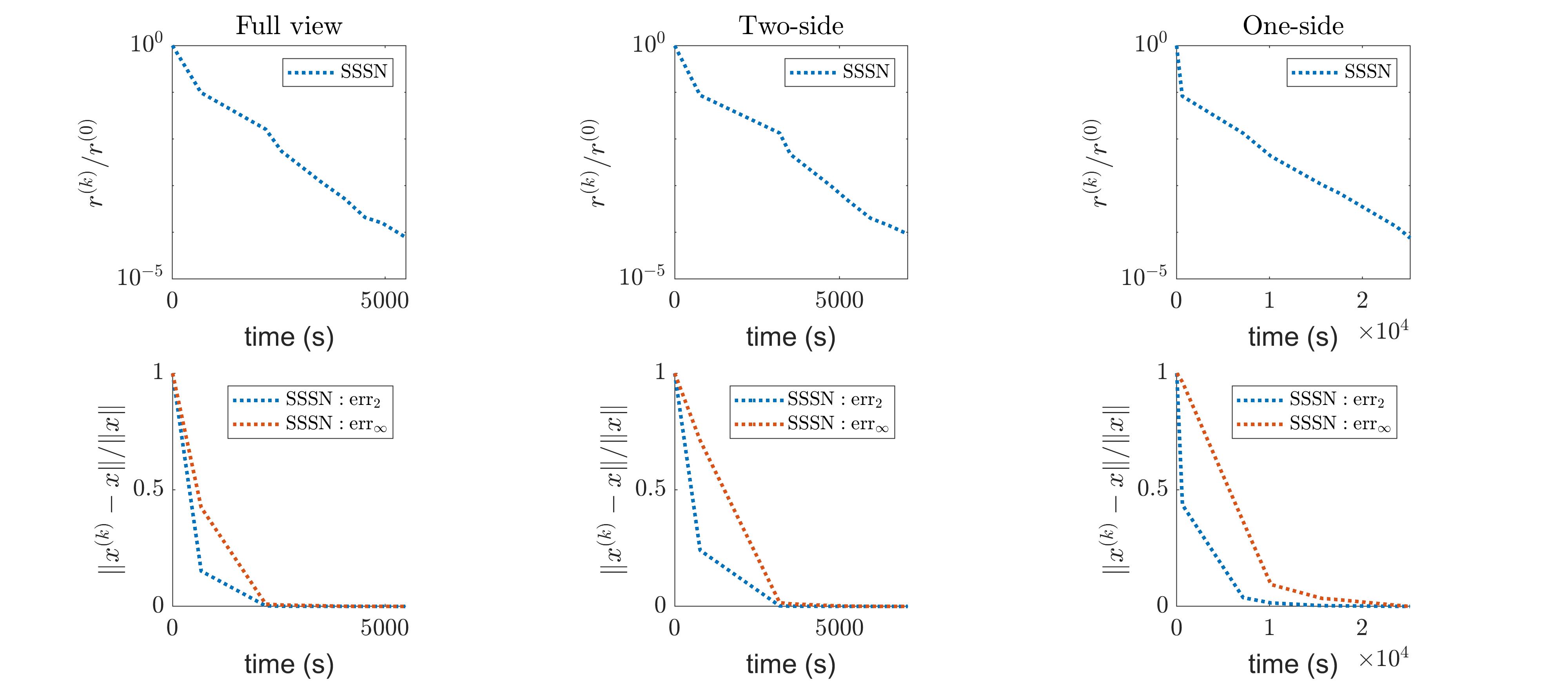}
    \caption{Test setting II (PAT): Comparison of relative residuals (top row) and relative errors (bottom row), both vs.\ CPU time, for different test configurations in the $256 \times 256$ pixel dimension (=medium-resolution) case: Full-view (left), two-sided (middle), and one-sided (right) sensor layout. Here, ``SSSN'' stands for our \ssstar Newton approach, i.e., Algorithm~\ref{AlgALM}.}
    \label{fig_C1}
\end{figure}

\begin{figure}[ht!]
    \centering
    \includegraphics[width=\textwidth]{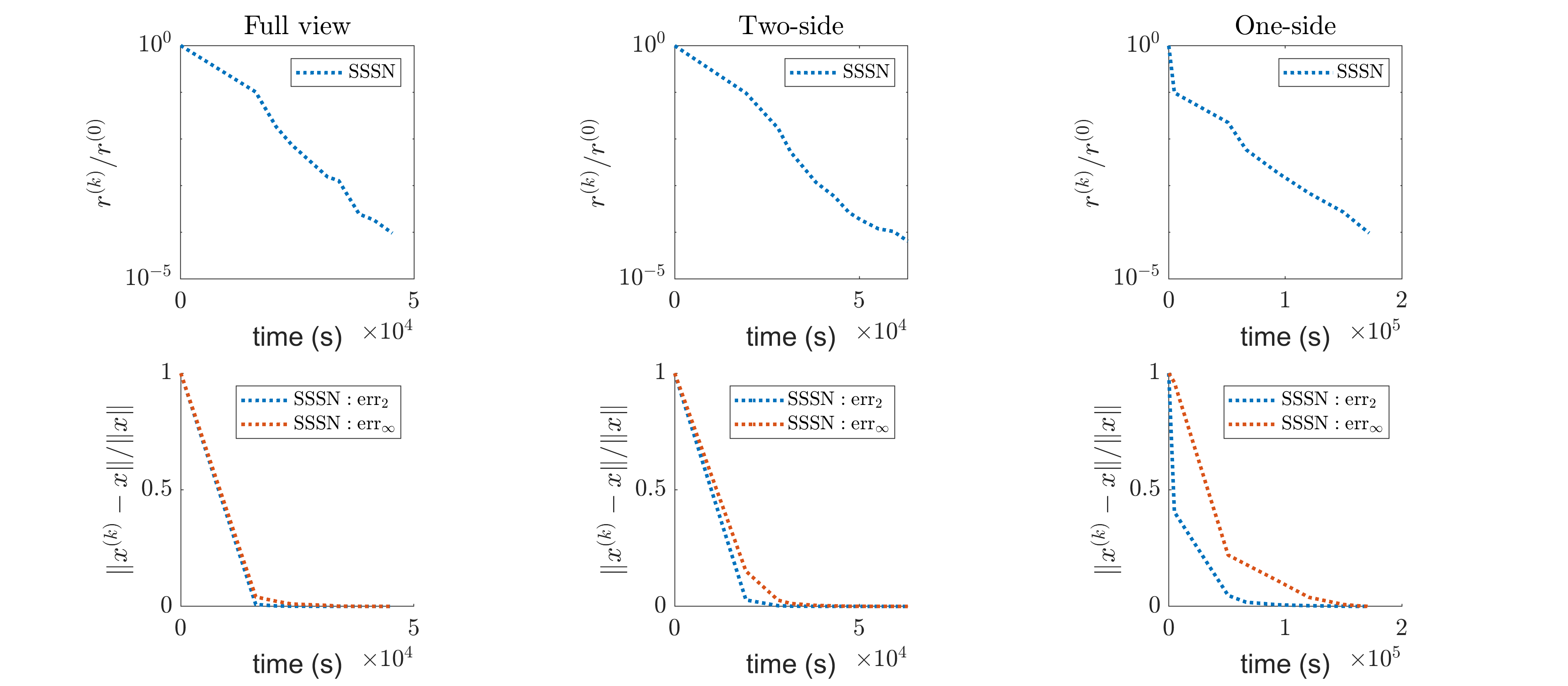}
    \caption{Test setting II (PAT): Comparison of relative residuals (top row) and relative errors (bottom row), both vs.\ CPU time, for different test configurations in the $512 \times 512$ pixel dimension (=high-resolution) case: Full-view (left), two-sided (middle), and one-sided (right) sensor layout. Here, ``SSSN'' stands for our \ssstar Newton approach, i.e., Algorithm~\ref{AlgALM}.}
    \label{fig_C2}
\end{figure}

\begin{table}[ht!]
\centering
\begin{tabular}{|c|c||c|c|c||}
\hline
&Sensor locations&full-view&two-sided&one-sides\cr
\hline
&$\rm it_{ALM}$&4&4&4\cr
SSSN&$\rm it_{SSSN}$&18&24&48\cr
$\eps_{\op{Opt}}=10^{-3}$&cg&59&92&423\cr
&time&35.2&42.2&135.9\cr
\hline
&$\rm it_{ALM}$&19&17&14\cr
SSSN&$\rm it_{SSSN}$&85&80&92\cr
$\eps_{\op{Opt}}=10^{-6}$&cg&208&225&805\cr
&time&111.1&104.2&252.6\cr
\hline
&it&75&248&882\cr
CP&cg&528&1470&4690\cr
$\eps_{\op{Opt}}^{\op{CP}}=10^{-3}$&time&165.7&425.9&1367.9\cr
\hline
&it&136&125&132\cr
ADMM&cg&3300&3750&3960\cr
$\eps_{\op{Opt}}^{\op{ADMM}}=10^{-3}$&time&850.3&952.0&1006.1\cr
\hline
\end{tabular}
\caption{Test setting II (PAT): Performance metrics of the considered methods for a representative test configuration with 128 × 128 pixels with precomputed matrix~$A$. Here, ``SSSN'' stands for our \ssstar Newton approach, i.e., Algorithm~\ref{AlgALM} in conjunction with Algorithm~\ref{AlgSubProbl}, and ``CP'' and ``ADMM'' for Algorithms~\ref{AlgChPo} and \ref{AlgADMM}, respectively. The number of iterations required by Algorithm~\ref{AlgALM} is denoted by $it_{\rm ALM}$, whereas $it_{\rm SSSN}$ is the total number of iterations performed by Algorithm~\ref{AlgSubProbl}. For ``CP''
and ``ADMM'', {\em it} denotes the number of iterations in the respective algorithm. For all three methods, {\em cg} and {\em time} report the total number of CG iterations and CPU time measured in seconds.}
\label{TabPAT_Comp_Blocks}
\end{table}

\end{document}

%% file: header.tex
\usepackage{amsmath}
\usepackage{amsthm}
\usepackage{amssymb}
\usepackage{cite}
\usepackage{url}
\usepackage{footmisc}
\usepackage{graphicx}
\usepackage{epstopdf}
\usepackage{chngcntr}
\usepackage{enumitem}
\usepackage{hhline}
\usepackage{array}
\usepackage{hyperref}
\usepackage{color}
\usepackage{multirow}
\usepackage{tabularx}
\usepackage{tikz}

\hypersetup{colorlinks=false}

\makeatletter
\let\@fnsymbol\@arabic
\makeatother

\tikzset{every picture/.style={line width=0.75pt}} 
\usetikzlibrary{patterns}

\tikzset{
    pattern size/.store in=\mcSize, 
    pattern size = 5pt,
    pattern thickness/.store in=\mcThickness, 
    pattern thickness = 0.3pt,
    pattern radius/.store in=\mcRadius, 
    pattern radius = 1pt
}
\makeatletter
\pgfutil@ifundefined{pgf@pattern@name@_7gn8i7dj0}{
    \pgfdeclarepatternformonly[\mcThickness,\mcSize]{_7gn8i7dj0}
    {\pgfqpoint{0pt}{0pt}}
    {\pgfpoint{\mcSize+\mcThickness}{\mcSize+\mcThickness}}
    {\pgfpoint{\mcSize}{\mcSize}}
    {
    \pgfsetcolor{\tikz@pattern@color}
    \pgfsetlinewidth{\mcThickness}
    \pgfpathmoveto{\pgfqpoint{0pt}{0pt}}
    \pgfpathlineto{\pgfpoint{\mcSize+\mcThickness}{\mcSize+\mcThickness}}
    \pgfusepath{stroke}
    }
}
\makeatother

\if@twoside
\else 
\fi

\numberwithin{equation}{section}
\counterwithin{figure}{section}
\counterwithin{table}{section}

\theoremstyle{plain}
\newtheorem{theorem}{Theorem}[section]

\newtheorem{proposition}[theorem]{Proposition}
\newtheorem{algorithm}  [theorem]{Algorithm}

\theoremstyle{definition}
\newtheorem{definition}{Definition}[section]

\theoremstyle{remark}
\newtheorem*{remark}{Remark}

\newcommand{\norm}[1]{\left\|#1\right\|}
\newcommand{\abs}[1]{\left\vert#1\right\vert}
\newcommand{\spr}[1]{\left\langle\,#1\,\right\rangle}
\newcommand{\kl}[1]{\left(#1\right)}
\newcommand{\Kl}[1]{\left\{#1\right\}}

\newcommand{\op}[1]{\operatorname{#1}}

\newcommand{\R}{\mathbb{R}} 

\newcommand{\mfunc}[1]{\text{#1} \,}
\def\div{\mfunc{div}} 

\newcommand{\xad}{x_\alpha^\delta}

\newcommand{\eps}{\varepsilon}

\newcommand{\Tad}{T_\alpha^\delta}

\newcommand{\xD}{x^\dagger}

\newcommand{\BV}{{\operatorname{BV}}}

\newcommand{\LtO}{{L_2(\Omega)}}

\newcommand{\prox}{\operatorname{prox}}

\newcommand{\tto}{\rightrightarrows}

\newcommand{\dist}[1]{{\rm dist}\kl{#1}}

\newcommand{\mv}{\,\mid\,}

\newcommand{\B}{{\cal B}}

\newcommand{\Sp}{{\mathcal S}}
\newcommand{\Z}{{\cal Z}}

\newcommand{\longsetto}[1]{\mathop{\longrightarrow}\limits^{#1}}
\newcommand{\skalp}[1]{\langle #1\rangle}

\newcommand{\xb}{\bar x}

\newcommand{\zb}{\bar z}
\newcommand{\ub}{\bar u}
\newcommand{\vb}{\bar v}

\newcommand{\zba}{{\bar z^\ast}}

\newcommand{\OO}{{\cal O}}
\newcommand{\argmin}{\mathop{\rm arg\,min}}

\newcommand{\ri}{{\rm ri\,}}

\newcommand{\gph}{\mathrm{gph}\,}

\newcommand{\myvec}[1]{\begin{pmatrix}#1\end{pmatrix}}
\newcommand{\SCD}{SCD\ }

\newcommand{\ssstar}{semismooth$^{*}$ }
\newcommand{\ee}[2]{{#1}^{(#2)}}

\newcommand{\rge}{{\rm rge\;}}

\renewcommand{\ker}{{\rm ker\;}}

\newcommand{\Mor}[1]{e_{#1 \norm{\cdot}_1}}
\newcommand{\Prox}[1]{{\rm Prox}_{#1 \norm{\cdot}_1}}
\newcommand{\lag}{{\cal L}}

\newlength{\myAlgBox}
\newcommand{\row}{\operatorname{row}}
\newcommand{\col}{\operatorname{col}}
